\documentclass[a4paper]{amsart}
\usepackage{amssymb}
\usepackage{amsmath,amsthm}
\usepackage{enumerate,paralist}
\usepackage{amstext}
\usepackage{psfrag}
\usepackage{dsfont}
\usepackage[dvips]{epsfig}
\usepackage[english]{babel}
\usepackage{color}
\usepackage{mathrsfs}
\usepackage{footnote}

\theoremstyle{plain}
\newtheorem{theorem}{Theorem}[section]

\newtheorem{lemma}[theorem]{Lemma}
\newtheorem{proposition}[theorem]{Proposition}
\theoremstyle{definition}
\newtheorem{definition}[theorem]{Definition}

\newtheorem{remark}[theorem]{Remark}

\renewcommand{\Re}{\ensuremath{{\rm Re\,}}}

\renewcommand{\leq}{\leqslant}
\renewcommand{\geq}{\geqslant}

\newcommand{\real}{\mathds R}
\newcommand{\rn}{{{\mathds R}^d}}

\newcommand{\ffi}{\varphi}
\newcommand{\eps}{\varepsilon}

\newcommand{\supp}{\mathop{\mathrm{supp}}\nolimits}
\newcommand{\id}{\mathop{\mathrm{Id}}\nolimits}

\def\tr{\mbox{\rm tr\hspace{.5mm}}}

\def\sbs{\subset}

\def\liml{\lim\limits}
\def\suml{\sum\limits}
\def\intl{\int\limits}
\def\prodl{\prod\limits}

 \def\vol{\text{\rm vol}}
\def\supp{\text{\rm supp}}

\def\Dom{\text{\rm Dom}}
\def\D{\text{\rm Dom}}
\def\Hess{\text{\rm Hess}}

\def\dim{\text{\rm dim}}
 \def\any{\forall}

\def\ffi{\varphi}

\def\cR{{\mathbb R}}
\def\cC{{\mathbb C}}

\def\<{\left\langle}
\def\>{\right\rangle}

\def \NN{\mathbb{N}}

\definecolor{kb}{rgb}{0.1,0.5,0.1}
\definecolor{yk}{rgb}{0.1,0.2,0.7}
\definecolor{ks}{rgb}{0.7,0.1,0.2}

\begin{document}

\title[Chernoff approximation of subordinate semigroups]{Chernoff approximation of subordinate semigroups and applications}

 % 	47D06   	One-parameter semigroups and linear evolution equations
 %   47D08   	Schrödinger and Feynman-Kac semigroups
 %   47D07   	Markov semigroups and applications to diffusion processes
 %  60J35     Transition functions, generators and resolvents


\thanks{I would like to  thank Ren\'{e} Schilling for  attracting my attention to the problem of Chernoff approximation of subordinate semigroups and  for  stimulating discussions in Bonn 2011. I also would like to thank Martin Fuchs for his  encouragement and support  of my researches. }

\author{Yana A. Butko}
\address{University of Saarland,  P.O. Box 15 11 50, D-66041 Saarbr\"{u}cken, Germany}
\email{yanabutko@yandex.ru, kinderknecht@math.uni-sb.de}

\date{\today}

\begin{abstract}
In this note the Chernoff Theorem is used to approximate evolution semigroups constructed  by the procedure of subordination.  The considered semigroups are subordinate to some original, unknown explicitly but already approximated by the same method, counterparts with respect to subordinators either with known transitional probabilities, or with known and bounded L\'{e}vy measure. 
 These results are applied  to obtain approximations of semigroups corresponding to  subordination of Feller  processes, and  (Feller type)  diffusions in  Euclidean spaces,  star graphs and Riemannian manifolds. The obtained approximations are based on explicitly given operators and hence can be used for direct calculations and computer modelling. In several cases the obtained approximations are given as  iterated integrals of elementary functions and  lead to representations of the considered semigroups by Feynman formulae.

\end{abstract}
\maketitle

 \tableofcontents

\section{Introduction}\label{sec:intro}
The problem to construct  a semigroup $\left(e^{tL}\right)_{t\geq0}$ with  a given generator $L$ (on a given Banach space) is very important for many applications. In one hand,  the semigroup  $e^{tL}$ allows to solve an initial (or  initial-boundary) value problem for the corresponding evolution equation $\frac{\partial f}{\partial t}=Lf$. On the other hand, the semigroup $e^{tL}$ defines the transition probability of the underlying stochastic process. One of the ways to construct strongly continuous semigroups is given by the procedure of subordination. From  two ingredients: an original strongly continuous contraction semigroup $(T_t)_{t\geq0}$ and a convolution semigroup $(\eta_t)_{t\geq0}$ supported by $[0,\infty)$  (see all   definitions in Sec.~2) this procedure produces  the strongly continuous contraction semigroup $(T^f_t)_{t\geq0}$  with  $T^f_t:=\int_0^\infty T_s\eta_t(ds)$. If the semigroup $(T_t)_{t\geq0}$ corresponds to a stochastic process $(X_t)_{t\geq0}$, then  subordination is a random time-change of $(X_t)_{t\geq0}$ by an independent increasing L\'{e}vy process (subordinator) related to $(\eta_t)_{t\geq0}$. If $(T_t)_{t\geq0}$ and  $(\eta_t)_{t\geq0}$ both are known explicitly, so is $(T^f_t)_{t\geq0}$. But if, e.g., $(T_t)_{t\geq0}$ is not known, neither $(T^f_t)_{t\geq0}$ itself, nor even the generator of $(T^f_t)_{t\geq0}$ are known explicitly any more.

Sure, if a desired semigroup is unknown, it  must be  approximated.  One of the methods to approximate evolution semigroups is based on the Chernoff Theorem. This theorem provides conditions for a family (just a family, not a semigroup!) of bounded linear operators $(F(t))_{t\geq0}$ to  approximate a  semigroup $(e^{tL})_{t\geq0}$ with a given generator $L$  via the formula $e^{tL}=\lim\limits_{n\to\infty}\left[  F(t/n) \right]^n$. This formula is called  \emph{Chernoff approximation of the semigroup $e^{tL}$ by the family $(F(t))_{t\geq0}$} and this  family is called  \emph{Chernoff equivalent} to the  semigroup $e^{tL}$. The most important condition of the Chernoff Theorem  is the coincidence of the derivative of $F(t)$ at $t=0$  with the generator $L$.

Chernoff approximation has the following advantage: if the family $(F(t))_{t\geq0}$ is given explicitly,  the expressions $\left[  F(t/n) \right]^n$ can be directly used for  calculations and hence for approximation of solutions of corresponding evolution equations, for computer modelling of considered dynamics, for approximation of transition probabilities of underlying stochastic processes and hence for simulation of these processes.  Moreover, if all operators $F(t)$ are integral operators with elementary kernels, the identity  $e^{tL}=\lim\limits_{n\to\infty}\left[  F(t/n) \right]^n$ leads to  representation of the semigroup $e^{tL}$ by  limits of  $n$-fold iterated  integrals of elementary functions when $n$ tends to infinity.  Such representations are called \emph{Feynman formulae}. The limits in Feynman formulae usually coincide with   functional (or path) integrals with respect to probability measures (Feynman-Kac formulae) or with respect to Feynman pseudomeasures (Feynman path integrals).  Feynman-Kac formulae allow to investigate the considered evolution, e.g., by the method of Monte Carlo, Feynman path integrals are an important tool in quantum physics. Therefore,  representations of evolution semigroups by Feynman formulae provide additional advantages and, in particular,   allow to establish new Feynman-Kac formulae, to investigate relations between different functional integrals,  to  develop the mathematical apparatus of Feynman path integrals and to calculate functional  integrals numerically.

 One further advantage of Chernoff approximation is the fact that this method is applicable for a broad class of evolution semigroups corresponding to different types of dynamics on different geometrical structures   (see, e.g. \cite{MR2863557}, \cite{MR2423533}, \cite{Yana_Technomag14},      \cite{MR2766564}, \cite{MR2963683}, \cite{MR1927359}, \cite{MR2276523} and references therein). This method, however, has never before been applied to approximation of subordinate semigroups. And the reason is already mentioned above:  if the original semigroup $(T_t)_{t\geq0}$ is not known explicitly then  the  generator of the  subordinate  to $(T_t)_{t\geq0}$ semigroup $(T^f_t)_{t\geq0}$   is not known  explicitly too. This impedes the construction of a family $(F(t))_{t\geq0}$  with a prescribed  (but unknown explicitly) derivative at $t=0$. This difficulty is overwhelmed in the present note by construction of families $(\mathcal{F}(t))_{t\geq0}$ and $(\mathcal{F}_\mu(t))_{t\geq0}$ (they are defined in Sec.~3) which incorporate approximations of the generator of $(T^f_t)_{t\geq0}$   itself.
 
 In this note  the semigroup $(T^f_t)_{t\ge0}$ subordinate   to a given semigroup $(T_t)_{t\ge0}$ with respect to a given subordinator is considered.  It is assumed  that the  subordinator  is known explicitly, i.e. either its transition probability is known, or its L\'{e}vy measure is known and bounded. Chernoff  approximations of the subordinate   semigroup $(T^f_t)_{t\ge0}$ are constructed in the case, when the semigroup $(T_t)_{t\ge0}$ is not known explicitly but is  already (Chernoff) approximated by a given  family $(F(t))_{t\ge0}$. These general results are applied further to obtain approximations of semigroups corresponding to  subordination of Feller  processes, and  (Feller type)  diffusions in  Euclidean spaces,  star graphs and Riemannian manifolds. Some of the obtained Chernoff approximations turn out to be Feynman formulae.

\section{Notations and Preliminaries} \label{sec:prelim}
\subsection{Subordination of semigroups}\label{subsec:Subordin}

We follow the exposition of the book \cite{MR1873235} %Jacob......
in this subsection.
Let  $(X,\|\cdot\|_X)$ be  a  Banach space,  $\mathcal{L}(X)$ be the space of all continuous linear operators on $X$ equipped with the  topology of strong operator convergence,  $\|\cdot\|$ denote the operator norm on $\mathcal{L}(X)$ and $\id$ be the identity operator in $X$. The symbol  $\D(L)$ denotes the domain of a linear operator $L$ in $X$, i.e.  $L: \D(L) \to X$.  A one-parameter family $(T_t)_{t\geq0}$ of bounded linear operators $T_t\,:X\to X$ is called a \emph{strongly continuous semigroup},  if  $T_0=\id$, $T_{s+t}=T_s\circ T_t$ for all $s,t\geq 0$ and $\lim_{t\to0}\|T_t\varphi-\varphi\|_X=0$ for all $\varphi\in X$. 
The semigroup $(T_t)_{t\geq 0}$ is called a \emph{contraction semigroup} if $\|T_t\|\leq1$  for all $t\geq 0$.  A family of bounded Borel measures $(\eta_t)_{t\geq 0}$ is called \emph{convolution semigroup} on $\cR^d$ if $\eta_t(\cR^d)\leq1$ for all $t\geq0$,  $\eta_s*\eta_t=\eta_{s+t}$ for all $s, t\geq0$, $\eta_0=\delta_0$ and $\eta_t\rightarrow\delta_0$ weakly (cf. Def. 3.6.1, Theo 2.3.7 and Lem. 3.6.2 of \cite{MR1873235}) %%vaguely 
as $t\to0$, where $\delta_0$ is the Dirac delta-measure concentrated at zero, and $\eta_s*\eta_t$ is the convolution of two measures. Each convolution semigroup $(\eta_t)_{t\geq0}$ on $\cR^d$ defines a strongly  continuous contraction semigroup $(S^\eta_t)_{t\geq 0}$ on the Banach space $C_\infty(\cR^d, \|\cdot\|_\infty)$ of continuous on $\cR^d$ functions vanishing at infinity equipped  with the supremum-norm $\|\cdot\|_\infty$ by the rule 
\begin{equation}\label{eq:S^eta}
S^\eta_t\ffi(x):=\intl_{\cR^d}\ffi(x+y)\eta_t(dy),\quad\quad \any\,\ffi\in C_\infty(\cR^d, \|\cdot\|_\infty).
\end{equation}
Let $(\eta_t)_{t\geq 0}$ be a convolution semigroup of measures on $\cR$. It is said to be supported by $[0,\infty)$ if $\supp\, \eta_t\sbs[0,\infty)$ for all $t\geq 0$.  Each convolution semigroup $(\eta_t)_{t\geq 0}$ supported by $[0,\infty)$ corresponds to a   \emph{Bernstein function} $f$ via the Laplace transform $\mathcal{L}$:  $\mathcal{L}[\eta_t](x)=e^{-tf(x)}$ for all  $x>0$ and $t>0$. Each Bernstein function $f$ is uniquely defined by a triplet $(\sigma,\lambda,\mu)$, where constants $\sigma,\lambda\geq0$ and $\mu$ is a Radon measure  on $(0,\infty)$ with $\int_{0+}^\infty\frac{s}{1+s}\mu(ds)<\infty$,
 through the representation
\begin{equation}\label{eq:BernsteinFunc}
f(z)=\sigma+\lambda z+\intl_{0+}^\infty (1-e^{-sz})\mu(ds),\quad\quad\any\,z\,: \Re z\geq 0.
\end{equation}
Note that   $\eta_t(\cR)=1$, $\forall\, t\geq0$, if and only if $\sigma=0$ (i.e., there is no "killing", cf. \cite{MR3156646}).

Let $(T_t)_{t\geq 0}$ be a strongly continuous contraction semigroup on a Banach space $(X,\|\cdot\|_X)$ and $(\eta_t)_{t\geq 0}$ be a convolution semigroup on $\cR$ supported by $[0,\infty)$ with the associated Bernstein function $f$. The family of operators $(T^f_t)_{t\geq 0}$ defined on $X$ by the Bochner integral
\begin{equation}\label{eq:subordination}
T^f_t\ffi:=\intl_0^\infty T_s\ffi\,\eta_t(ds)
\end{equation}
is again a strongly continuous contraction semigroup on $X$. The semigroup $(T^f_t)_{t\geq 0}$ is called \emph{subordinate} (in the sense of Bochner)  to $(T_t)_{t\geq 0}$ with respect to $(\eta_t)_{t\geq 0}$.

Each convolution semigroup $(\eta_t)_{t\geq 0}$ corresponds to a L\'{e}vy process $(\xi_t)_{t\geq 0}$  via $\ffi*\eta_t(x)=\mathbb{E}[\ffi(x-\xi_t)]$. If a convolution semigroup $(\eta_t)_{t\geq 0}$ is supported by $[0,\infty)$ then the corresponding L\'{e}vy process  $(\xi_t)_{t\geq 0}$ has non-decreasing paths almost surely and is called \emph{subordinator}. Such processes can be used for a time-change of another processes. Namely, if $(X_t)_{t\geq 0}$ is a (decent) Markov   process then  the \emph{subordinate process} $(X_{\xi_t})_{t\geq 0}$  ($X_{\xi_t}(\omega):=X_{\xi_t(\omega)}(\omega)$) is again a (decent) Markov   process.  E.g., if $(X_t)_{t\geq 0}$ is a Feller process  then $(X_{\xi_t})_{t\geq 0}$ is again  a Feller process (see Definition in Sect~4.1).
If $(T_t)_{t\geq 0}$ is the strongly continuous contraction semigroup corresponding to $(X_t)_{t\geq 0}$, i.e. $T_t\ffi(x)=\mathbb{E}[\ffi(x+X_t)]$,  and $(\eta_t)_{t\geq 0}$ is the convolution semigroup of the subordinator $(\xi_t)_{t\geq 0}$  then the defined in \eqref{eq:subordination}  subordinated semigroup $(T^f_t)_{t\geq 0}$  corresponds to the subordinate   process $(X_{\xi_t})_{t\geq 0}$. Many interesting processes  (see $\S 4.4$ of \cite{MR2042661}) are obtained from the Brownian motion via subordination.

Let  $(T_t)_{t\geq 0}$  be  a strongly continuous semigroup on a Banach space $(X,\|\cdot\|_X)$. Its  generator $L$ %%of $(T_t)_{t\geq 0}$ 
is defined by
 \begin{align*}
 &L\varphi:=\liml_{t\searrow0}\frac{T_t\varphi-\varphi}{t} \quad\text{with the domain}\\
 &
  \D(L):=\bigg\{ \varphi\in X \,\bigg| \quad\liml_{t\searrow0}\frac{T_t\varphi-\varphi}{t}\quad\text{\rm exists in } X \bigg\}.
 \end{align*}
Consider, in particular,   the given by \eqref{eq:S^eta} operator semigroup $(S^\eta_t)_{t\geq q0}$ on $C_\infty(\cR)$ corresponding to a supported by $[0,\infty)$ convolution semigroup $(\eta_t)_{t\geq 0}$. Assume that the corresponding Bernstein function is given by a triplet $(0,0,\mu)$. Then the generator $(L^\eta,\Dom(L^\eta))$ of  $(S^\eta_t)_{t\geq0}$   has the following properties: $C^\infty_c(\cR)\subset \Dom(L^\eta)$ and  for all $\ffi\in \Dom(L^\eta)$  
\begin{align}\label{eq: L^eta}
L^\eta\ffi(x)=\intl_{0+}^\infty (\ffi(x+s)-\ffi(x))\mu(ds).
\end{align} 
%%%Moreover,  for each $\ffi\in C_c(\cR\setminus \{ 0  \})$  one has $\liml_{t\to 0}\frac1t\int_0^\infty \ffi(s)\eta_t(ds)=\liml_{t\to 0}\frac1t S^\eta_t\ffi(0)=\intl_{0+}^\infty \ffi(s)\mu(ds)$ (cf. Lemma 2.16 of \cite{MR3156646}).
Let now $(T_t)_{t\geq 0}$  be a strongly continuous contraction semigroup with the generator $(L,\D(L))$ and $f$ be a Bernstein function  given by  the representation \eqref{eq:BernsteinFunc}  with associated convolution semigroup $(\eta_t)_{t\geq 0}$ supported by $[0,\infty)$. Then   $\D(L)$ is a core for  the generator $L^f$ of the subordinate semigroup $(T^f_t)_{t\geq 0}$  and for $\ffi\in\D(L)$ the operator $L^f$ has the representation
\begin{equation}\label{eq:L^f}
L^f\ffi=-\sigma\ffi+\lambda L\ffi+\intl_{0+}^\infty (T_s\ffi-\ffi)\mu(ds).
\end{equation} 
Note  that if a linear subspace $D\sbs X$ is a core for $L$, then $D$ is also a core for $L^f$ (see \cite{MR1739520}, Prop. 32.5, p. 215).%%Sato.

For each convolution semigroup $(\eta_t)_{t\geq 0}$  the corresponding operator semigroup $(S^\eta_t)_{t\geq 0}$ extends to a contraction semigroup $(\bar {S}^\eta_t)_{t\geq0}$ on the space $B_b(\cR^d)$ of all bounded Borel functions on $\cR$. This semigroup belongs to the class of   \emph{strong Feller} semigroups\footnote{ A semigroup $(T_t)_{t\geq 0}$ is called  \emph{strong Feller } if it is a   positivity preserving sub-Markovian semigroup on $B_b(\cR^d)$ (i.e. $0\leq T_t\ffi\leq 1$ for all $\ffi\in B_b(\cR^d)$ with $0\leq\ffi\leq1$), all operators $T_t$ map  $B_b(\cR^d)$ into the space of continuous bounded functions  $C_b(\cR^d)$ and $(T_t)_{t\geq 0}$ is a strongly continuous semigroup on $C_\infty(\cR^d)$ (cf. \cite{MR3156646}, \cite{MR1873235}).} if and only if all the measures $\eta_t$ admit densities of class $L^1(\cR^d)$ with respect to the Lebesgue measure (cf. Examle 4.8.21 of \cite{MR1873235}).  
One may consider a strong Feller semigroup $(\bar {S}^\eta_t)_{t\geq0}$ as a semigroup on the space $C_b(\cR^d)$ of all bounded continuous  functions and define its $C_b$-generator $(\bar {L}^\eta,\Dom(\bar {L}^\eta))$ for each $x\in \cR^d$ by 
\begin{align*}
 &\bar {L}^\eta\varphi(x):=\liml_{t\searrow0}\frac{\bar {S}^\eta_t\varphi(x)-\varphi(x)}{t} \quad\text{with the domain }\quad \D(\bar {L}^\eta):=\\
 &
  =\bigg\{ \varphi\in C_b(\cR^d) \,\bigg| \,\liml_{t\searrow0}\frac{\bar {S}^\eta_t\varphi(x)-\varphi(x)}{t}\, \,\text{\rm exists uniformly on compact subsets of } \cR^d \bigg\}.
 \end{align*}
The operator $(\bar {L}^\eta,\Dom(\bar {L}^\eta))$ in $C_b(\cR^d)$  is an extension of  the generator $(L^\eta,\Dom(L^\eta))$ of the semigroup $({S}^\eta_t)_{t\geq0}$ on $C_\infty(\cR^d)$  and in particular $C^2_b(\cR^d)\subset \Dom(\bar {L}^\eta)$ (cf.  Example 4.8.26 of \cite{MR1873235}).

\subsection{The Chernoff Theorem and Feynman formulae}\label{subsec:Chernoff}

Consider an  evolution equation $\frac{\partial f}{\partial t}=Lf$. If $L$ is the generator of a strongly continuous semigroup  $(T_t)_{t\geq 0}$ on a Banach space $(X,\|\cdot\|_X)$, then the (mild) solution of the Cauchy problem for this equation with the initial value $f(0)=f_0\in X$ is given by $f(t)=T_tf_0$ for all $f_0\in X$. Therefore, to solve the evolution equation $\frac{\partial f}{\partial t}=Lf$ means to construct a semigroup $(T_t)_{t\geq 0}$ with the given generator $L$. If  the desired semigroup is not known explicitly it can be  approximated. One of the tools to approximate  semigroups is based on the Chernoff theorem \cite{MR0231238}. 
%%(here we  present the version of Chernoff's theorem  given in  \cite{MR1927359}).
\begin{theorem}[Chernoff]\label{thm:Chernoff}
Let $X$ be a Banach space, $F:[0,\infty)\to{\mathcal{L}}(X)$ be a (strongly) continuous mapping such that $F(0) =\id$  and $\|F(t)\|\le e^{at}$ for some  $ a\in [0, \infty)$ and all $t \geq  0$. Let  $D$ be a linear subspace of  $\D(F'(0))$ such that the restriction of the operator  $F'(0)$ to this subspace is closable. Let $(L, \D(L))$ be this closure. If  $(L, \D(L))$ is the generator of a strongly continuous semigroup  $(T_t)_{t \geq  0}$, then for any $t_0 >0$ and any $\ffi\in X$ the sequence
 $(F(t/n))^n \ffi)_{n \in {\mathds N}}$ converges to $T_t\ffi$ as $n\to\infty$   uniformly with respect to $t\in[0,t_0],$ i.e.
 \begin{equation}\label{FF}
 T_t = \lim_{n\to\infty}\left[F(t/n)\right]^n
 \end{equation}
in the sence of the strong operator convergence locally uniformly with respect to  $t>0$.
\end{theorem}
Here the derivative at the origin of a function $F:[0,\eps) \to \mathcal{L}(X)$, $\eps > 0$, is a linear mapping $F'(0): \D(F'(0)) \to X$ such that
\begin{equation*}
F'(0)\ffi := \lim_{t \searrow 0}\frac{F(t)\ffi -F(0)\ffi}{t},
\end{equation*}
where $\D(F'(0))$ is the vector space of all elements $\ffi \in X$ for which the above limit exists. A family of operators $(F(t))_{t \geq  0}$ suitable for the formula \eqref{FF}, i.e.   satisfying  all the assertions of the Chernoff  theorem with respect to the semigroup $(T_t)_{t \geq  0}$, is called \emph{Chernoff equivalent} to this semigroup. The equality \eqref{FF} is called \emph{Chernoff approximation of the semigroup $(T_t)_{t \geq  0}$ by the family $(F(t))_{t \geq  0}$}.
 In many cases the operators $F(t)$ are integral operators and, hence, we have a limit of iterated integrals on the right hand side of the equality \eqref{FF}. In this setting it is called  \emph{Feynman formula}.
\begin{definition}
A Feynman formula is a representation of a solution of an initial (or initial-boundary) value problem for an evolution equation (or, equivalently, a representation of the semigroup solving the problem) by a limit of $n$-fold iterated integrals of some functions as $n\to\infty$.
\end{definition}
Richard Feynman was the first who considered representations of solutions of evolution equations by limits of iterated integrals  (\cite{MR0026940}, \cite{MR0044379}).  He has, namely,  introduced  a functional (path) integral  for solving Schr\"{o}dinger equation. And this integral was defined exactly as a  limit of  iterated finite dimensional integrals. Representations of  the solution of an initial (or initial-boundary) value problem for an evolution equation (or, equivalently, a representation of the semigroup resolving the problem) by  functional (path) integrals are   called nowadays \emph{Feynman--Kac formulae}.  It is a usual situation that  limits in  Feynman formulae coincide with (or in some cases define) certain functional integrals with respect to probability measures or Feynman type pseudomeasures on a set  of paths of a physical system.  Hence, the iterated integrals in   Feynman formulae for some problem give approximations to  functional integrals in the Feynman-Kac formulae representing the solution of the same problem. These approximations in many cases contain only elementary functions as integrands and, therefore, can be used for direct calculations and simulations.

In the sequel we need also the following results  of  papers \cite{Yana_Technomag14} and \cite{MR2999096}:
\begin{theorem}[Theo. 5.1 of \cite{MR2999096}]\label{TH5.1}
Let $X$ be a Banach space with a norm $\|\cdot\|_X$. Let
$(T_k(t))_{t\geq 0}$, $k=1,\ldots,m,$   be  strongly continuous semigroups on $X$ with
generators $(L_k,\D(L_k))$ respectively.
Assume that $L=L_1+\cdots +L_m$ with   domain $\D(L)=\cap_{k=1}^m \D(L_k)$ is closable and that the closure is the generator of a strongly continuous
semigroup $(T(t))_{t\geq 0}$  on $X$.
Let  $(F_k(t))_{t\geq 0}$, $k=1,\ldots ,m,$ be families of operators in $X$  which are
Chernoff equivalent to the semigroups $(T_k(t))_{t\geq 0}$ respectively, i.e. for each $k\in\{1,\ldots ,m \}$ we have
$F_k(0)=\id$,  $\|F_k(t)\|\le e^{a_kt}$ for some $a_k>0$ and there is a set
$D_k\subset \D(L_k)$, which is a core for $L_k$, such that  $\lim_{t\to0}\big\|
\frac{F_k(t)\ffi-\ffi}{t}-L_k\ffi\big\|_X=0$ for each $\ffi\in D_k$.  Assume that there exists  a set $D\subset \cap_{k=1}^m D_k$ which is a core for $L$.
Then the family  $(F(t))_{t\geq 0}$, where $F(t)=F_1(t)\circ\cdots \circ F_m(t)$
is Chernoff equivalent to the semigroup $(T(t))_{t\geq 0}$ and hence the Chernoff approximation 
$$
T_t=\lim_{n\to\infty} \big[F(t/n)   \big]^n
$$
is valid in the sence of the strong operator convergence  locally uniformly w.r.t.  $t\geq  0$.
\end{theorem}
\begin{theorem}[Theo. 4 of \cite{Yana_Technomag14}]\label{thm:Technomag mult pert}
Let $Q$ be a metric space. Let $X=C_b(Q)$ or  $X=C_\infty(Q)$ with supremum norm $\|\cdot\|_\infty$. Let $(T(t))_{t\geq 0}$ be a strongly continuous semigroup on $X$ with generator $(L,\Dom(L))$. Let $A(\cdot)$ be a bounded strictly positive continuous function on $Q$. Let $(\widetilde{T}(t))_{t\geq 0}$ be a strongly continuous semigroup on $X$ with generator $(\widetilde{L},\Dom(L))$, where $\widetilde{L}\ffi(q):=A(q)L\ffi(q)$ for all $\ffi\in X$ and all $q\in Q$.
Let  $(F(t))_{t\geq 0}$  be a family of operators in $X$  which is Chernoff equivalent to the semigroup $(T(t))_{t\geq 0}$. Then the family $(\widetilde{F}(t))_{t\geq 0}$, where $\widetilde{F}\ffi(q):=(F(A(q)t)\ffi)(q)$ for all $\ffi\in X$ and all $q\in Q$, is Chernoff equivalent to the semigroup $(\widetilde{T}(t))_{t\geq 0}$ and hence the Chernoff approximation 
$$
\widetilde{T}_t=\lim_{n\to\infty} \big[\widetilde{F}(t/n)   \big]^n
$$
is valid in the sence of the strong operator convergence  locally uniformly w. r.t.  $t\geq  0$.
\end{theorem}

\section{Approximation of subordinate   semigroups}

\subsection{Case 1: transitional probabilities of subordinators are known}\label{subsec:knownsubordinators}

In this 
subsection we consider the semigroup $(T^f_t)_{t\geq 0}$ subordinate   to a given semigroup $(T_t)_{t\geq 0}$ with respect to a given convolution semigroup $(\eta_t)_{t\geq 0}$ associated to a Bernstein function $f$ defined by a triplet $(\sigma,\lambda,\mu)$.  We assume that the corresponding convolution semigroup $(\eta^0_t)_{t\geq 0}$ associated to the Bernstein function $f_0$ defined by the triplet $(0,0,\mu)$ is known explicitly and corresponds to a strong Feller semigroup $(\bar {S}^{\eta^0}_t)_{t\geq0}$.  This is the case of inverse Gaussian (including $1/2$-stable) subordinator,   Gamma subordinator and some others (see, e.g., \cite{MR3231629} for examples). We are interested in approximation of the subordinate   semigroup $(T^f_t)_{t\geq 0}$ when the semigroup $(T_t)_{t\geq 0}$ is not known explicitly but is  approximated by a given  family $(F(t))_{t\geq 0}$ which is Chernoff equivalent to $(T_t)_{t\geq 0}$.
\begin{theorem}\label{thm:ChFam_knownSub-r}
Let $(T_t)_{t\geq 0}$  be a strongly continuous contraction semigroup on a Banach space $(X,\|\cdot\|_X)$ with the generator $(L,\D(L))$. Let  $(F(t))_{t\geq 0}$ be a family of contractions on $(X,\|\cdot\|_X)$ which is Chernoff equivalent to $(T_t)_{t\geq 0}$, i.e. $F(0)=\id$,  $\|F(t)\|\le 1$ for all  $t\geq 0$ and there is a set $D\subset \D(L)$, which is a core for $L$, such that  $\lim_{t\to0}\big\|\frac{F(t)\ffi-\ffi}{t}-L\ffi\big\|_X=0$ for each $\ffi\in D$. Let $f$ be a Bernstein function  given by a triplet $(\sigma,\lambda,\mu)$ through  the representation \eqref{eq:BernsteinFunc}  with associated convolution semigroup $(\eta_t)_{t\geq 0}$ supported by $[0,\infty)$.
Let $(\eta^0_t)_{t\geq 0}$ be the (supported by $[0,\infty)$) convolution semigroup associated to the Bernstein function $f_0$ defined by the triplet $(0,0,\mu)$. Assume that the corresponding operator semigroup $(\bar {S}^{\eta^0}_t)_{t\geq0}$ is strong Feller.
 Let $m:(0,\infty)\to\mathds{N}_0$ be a monotone function with $m(t)\to\infty$ as $t\to0$\footnote{ One can take, e.g., $m(t):=\left[ 1/t  \right]=$  the largest integer $n\leq 1/t$. }. Let $(T^f_t)_{t\geq 0}$ be the semigroup subordinate   to $(T_t)_{t\geq 0}$ with respect to  $(\eta_t)_{t\geq 0}$ and $L^f$ be  its generator. Consider a family   $(\mathcal{F}(t))_{t\geq 0}$  of operators on $(X,\|\cdot\|_X)$ defined by $\mathcal{F}(0)=\id$ and
\begin{equation}\label{eq:mathcal(F)}
\mathcal{F}(t)\ffi=  e^{-\sigma t}\circ   F(\lambda t)\circ \mathcal{F}_0(t)\ffi,\quad t>0,\, \ffi\in X,
\end{equation}
with $\mathcal{F}_0(0)=\id$ and\footnote{For any bounded operator $B$ its zero degree $B^0$ is considered to be the identity operator.}
\begin{equation}\label{eq:F_0}
\mathcal{F}_0(t)\ffi=   \intl_{0+}^\infty \left[F( s/m(t))\right]^{m(t)}\ffi\,\eta^0_t(ds),\quad t>0,\, \ffi\in X.
\end{equation}
The family $(\mathcal{F}(t))_{t\geq 0}$ is Chernoff equivalent to the semigroup $(T^f_t)_{t\geq 0}$ and, hence
%%, the Feynman formula
$$
T^f_t \ffi=\lim_{n\to\infty} \big[\mathcal{F}(t/n)   \big]^n \ffi
$$
for all $\ffi\in X$ locally uniformly with respect to $t\geq  0$.
\end{theorem}

\begin{proof}
Let us prove that the family $(\mathcal{F}_0(t))_{t\geq 0}$ is Chernoff equivalent to the semigroup $(T^{f_0}_t)_{t\geq 0}$ subordinate   to $(T_t)_{t\geq 0}$ with respect to  $(\eta^0_t)_{t\geq 0}$. The generator $L^{f_0}$ of this semigroup for each $\ffi\in \Dom(L)$ is given by
\begin{equation}\label{eq:L^f0}
L^{f_0}\ffi=\intl_{0+}^\infty (T_s\ffi-\ffi)\mu(ds)
\end{equation}
and $D\subset \Dom(L)$ is a core for $L^{f_0}$.
 The statement of the Theorem   then follows immediately from Theorem \ref{TH5.1} since $F(\lambda t)$ is Chernoff equivalent to  $T_{\lambda t}\equiv e^{(t\lambda)L}\equiv e^{t(\lambda L)}$ for each $\lambda>0$. The proof, that  the family $(\mathcal{F}_0(t))_{t\geq 0}$ is Chernoff equivalent to the semigroup $(T^{f_0}_t)_{t\geq 0}$, is the subject of   the following five Lemmas.

\begin{lemma}\label{le:1}
%%Let us check first that $\mathcal{F}(t)\ffi\in X$ for each $\ffi\in X$ and 
Operators $\mathcal{F}_0(t)$ are  contractions on $X$ for all  $t>0$.
\end{lemma}

\begin{proof}
Taking into account that all operators $F(t)$ are contractions on $X$ and that $\eta^0_t(\cR)= 1$ one has for each $t>0$
\begin{align*}
\|\mathcal{F}(t)\ffi\|_X&=\left\| \,  \intl_{0+}^\infty \left[F( s/m(t))\right]^{m(t)}\ffi\,\eta^0_t(ds)  \right\|_X
\leq\intl_{0+}^\infty\left\|  \left[F( s/m(t))\right]^{m(t)}\ffi\right\|_X\,\eta^0_t(ds)\\
&
\leq \intl_{0+}^\infty 
\left\|  \left[F( s/m(t))\right]\right\|^{m(t)} \left\| \ffi\right\|_X\,\eta^0_t(ds)\leq \left\| \ffi\right\|_X.
\end{align*}
\end{proof}

\begin{lemma}\label{le:2}
The family $(\mathcal{F}_0(t))_{t\geq 0}$ is  strongly continuous.
\end{lemma}
\begin{proof}
Let us check first  that the family $(\mathcal{F}_0(t))_{t\geq 0}$ is strongly continuous at zero. For each $\ffi\in X$ we have
\begin{align*}
\liml_{t\to 0}&\| \mathcal{F}_0(t)\ffi-\ffi  \|_X=\liml_{t\to 0} \left\| \, \intl_{0+}^\infty \left[F( s/m(t))\right]^{m(t)}\ffi\,\eta^0_t(ds)-\ffi\right\|_X \\
&
\leq\liml_{t\to 0} \left\| \, \intl_{0+}^\infty \left[F( s/m(t))\right]^{m(t)}\ffi\,\eta^0_t(ds)-T^{f_0}_t\ffi\right\|_X+
\liml_{t\to 0} \left\| \,  T^{f_0}_t\ffi-\ffi\right\|_X\\
&
=\liml_{t\to 0} \left\| \, \intl_{0+}^\infty \left(\left[F( s/m(t))\right]^{m(t)}\ffi- T_s\ffi\right)\,\eta^0_t(ds)\right\|_X\\
&
\leq \liml_{t\to 0} \intl_{0+}^\infty  \left\| \left[F( s/m(t))\right]^{m(t)}\ffi- T_s\ffi\right\|_X\,\eta^0_t(ds)\\
&
\leq \liml_{t\to 0} \intl_{0+}^1  \left\| \left[F( s/m(t))\right]^{m(t)}\ffi- T_s\ffi\right\|_X\,\eta^0_t(ds) +\liml_{t\to 0} 2\|\ffi\|_X \intl_{1}^\infty  \eta^0_t(ds)\\
&
\leq \liml_{t\to 0} \sup\limits_{s\in[0,1]}\left\| \left[F( s/m(t))\right]^{m(t)}\ffi- T_s\ffi\right\|_X\\
&
=0
\end{align*}
since  the convergence of $\left\| \left[F( s/m(t))\right]^{m(t)}\ffi- T_s\ffi\right\|_X$ to zero as $t\to0$ is uniform w.r.t. $s$ on compact intervals due to the Chernoff Theorem and since $\eta^0_t$ weakly converges to the Dirac delta-measure $\delta_0$ as $t\to0$.

Let us now check the strong continuity of the family $(\mathcal{F}_0(t))_{t\geq 0}$ at a point $t_0>0$.
\begin{align*}
\liml_{t\to t_0}& \left\| \,\mathcal{F}_0(t)\ffi-\mathcal{F}_0(t_0)\ffi\right\|_X\\
&
=\liml_{t\to t_0}\left\| \,  \intl_{0+}^\infty \left[F( s/m(t))\right]^{m(t)}\ffi\,\eta^0_t(ds) - \intl_{0+}^\infty\left[F( s/m(t_0))\right]^{m(t_0)}\ffi  \,\eta^0_{t_0}(ds)                  \right\|_X\\
&
\leq \liml_{t\to t_0} \left\| \,\intl_{0+}^\infty \left[F( s/m(t))\right]^{m(t)}\ffi  \,[\eta^0_t-\eta^0_{t_0}](ds)\right\|_X   + \\
&
\quad + \liml_{t\to t_0} \intl_{0+}^\infty\left\| \,\left[F( s/m(t))\right]^{m(t)}\ffi - \left[F( s/m(t_0))\right]^{m(t_0)}\ffi  \right\|_X\,\eta^0_{t_0}(ds)
\end{align*}
Choose $\eps>0$ such that $[t_0-\eps, t_0+\eps]\subset (0,\infty)$ and consider $t\in[t_0-\eps, t_0+\eps]$. Let  $M$ be the maximum of $m(t)$ on the segment $[t_0-\eps, t_0+\eps]$, hence $M<\infty$. Let further  $t>t_0$, i.e. $m(t)\leq m(t_0)$. The opposite situation can be considered similarly. Denote $F:=F( s/m(t))$, $F_0:=F( s/m(t_0))$, $m:=m(t)$ and $m_0:=m(t_0)$. Then
\begin{align*}
&\left\| \,\left[F( s/m(t))\right]^{m(t)}\ffi - \left[F( s/m(t_0))\right]^{m(t_0)}\ffi  \right\|_X=
\left\|   F^m\ffi-F_0^{m_0}\ffi   \right\|_X \\
&
\leq      \left\|   F^m\ffi-F_0^{m}\ffi   \right\|_X + \left\|F^{m}_0\ffi - F^{m_0}_0\ffi    \right\|_X   \\
&
\leq \left\|F^{m-1}\circ F\ffi - F^{m-1}\circ F_0\ffi   \right\|_X + \left\|F^{m-1}\circ F_0\ffi - F^{m-1}_0\circ F_0\ffi    \right\|_X + \left\|F^{m}_0\ffi - F^{m_0}_0\ffi    \right\|_X \\
&
\leq \left\|F\ffi - F_0\ffi   \right\|_X + \left\|F^{m-1}(F_0\ffi) - F^{m-1}_0(F_0\ffi)    \right\|_X+ \left\|F^{m_0-m}_0\ffi - \ffi    \right\|_X\leq \ldots \\ 
&
\leq \sum\limits_{k=0}^m \left\|F(F_0^{k-1}\ffi) - F_0(F_0^{k-1}\ffi)    \right\|_X+ \left\|F^{m_0-m}_0\ffi - \ffi    \right\|_X\\ 
&
\leq \sum\limits_{k=0}^M \left\|F(F_0^{k-1}\ffi) - F_0(F_0^{k-1}\ffi)    \right\|_X+ \left\|F^{m_0-m}_0\ffi - \ffi    \right\|_X
\end{align*}
For any $\delta>0$ choose $R=R(\delta)$ so that $\int\limits_{R}^\infty\eta^0_{t_0}(ds)<\delta$.    Then
\begin{align*}
&\liml_{t\to t_0} \intl_{0+}^\infty\left\| \,\left[F( s/m(t))\right]^{m(t)}\ffi - \left[F( s/m(t_0))\right]^{m(t_0)}\ffi  \right\|_X\,\eta^0_{t_0}(ds)\\
&
\leq \liml_{t\to t_0} \intl_{0+}^R\left\| \,\left[F( s/m(t))\right]^{m(t)}\ffi - \left[F( s/m(t_0))\right]^{m(t_0)}\ffi  \right\|_X\,\eta^0_{t_0}(ds)\\
&
\quad+\liml_{t\to t_0} \intl_{R}^\infty\left\| \,\left[F( s/m(t))\right]^{m(t)}\ffi - \left[F( s/m(t_0))\right]^{m(t_0)}\ffi  \right\|_X\,\eta^0_{t_0}(ds)\\
&
\leq  \liml_{t\to t_0}\sup\limits_{s\in[0,R]}\left\| \,\left[F( s/m(t))\right]^{m(t)}\ffi - \left[F( s/m(t_0))\right]^{m(t_0)}\ffi  \right\|_X+ 2\|\ffi\|_X\intl_{R}^\infty\eta^0_{t_0}(ds)\\
&
\leq \sum\limits_{k=0}^M \liml_{t\to t_0} \sup\limits_{s\in[0,R]}\left\| \,F( s/m(t))\left[F( s/m(t_0))\right]^{k-1}\ffi- F( s/m(t_0))\left[F( s/m(t_0))\right]^{k-1}\ffi   \right\|_X  \\
&
\quad  +\liml_{t\to t_0} \sup\limits_{s\in[0,R]}\left\| \, \left[F( s/m(t_0))\right]^{m(t_0)-m(t)}\ffi -\ffi    \right\|_X + 2\|\ffi\|_X\delta\\
&
\leq 2\|\ffi\|_X\delta
\end{align*}
since the family $(F(t))_{t\geq0}$ is strongly continuous. The above inequalities hold for any $\delta>0$, hence one has
$$
\liml_{t\to t_0} \intl_{0+}^\infty\left\| \,\left[F( s/m(t))\right]^{m(t)}\ffi - \left[F( s/m(t_0))\right]^{m(t_0)}\ffi  \right\|_X\,\eta^0_{t_0}(ds)=0.
$$
Due to the weak convergence of $\eta^0_t $ to $\eta^0_{t_0}$ and due to continuity and boundness of the integrand as a function of $(s,t)$ one has also 
\begin{align*}
\liml_{t\to t_0} \left\| \,\intl_{0+}^\infty \left[F( s/m(t))\right]^{m(t)}\ffi  \,[\eta^0_t-\eta^0_{t_0}](ds)\right\|_X =0.
\end{align*}
This ends the proof. 
 \end{proof}  %%% of the lemma

\begin{lemma}\label{le:estimate of Psi}
For a fixed $\ffi\in D$ define the function $\Psi_t:[0,\infty)\to [0,\infty)$ by $\Psi_t(s):=\left\| F^{m(t)}(s/m(t))\ffi- T_s \ffi  \right\|_X$. For each $t>0$ and each $s>0$ the following estimate holds:
\begin{align*}\label{eq:Psi-estimate}
&\frac{\Psi_t(s)}{s}\leq \left\| \frac{T_s\ffi - \ffi}{s} - L\ffi   \right\|_X + \left\| \frac{F(s/m(t))\ffi - \ffi}{s/m(t)} - L\ffi   \right\|_X +\\
&
+\left\|  \left(\frac{1}{m(t)}\left[ F^{m(t)-1}(s/m(t))+F^{m(t)-2}(s/m(t))+\ldots+F(s/m(t)) +\id \right]  -\id    \right)L\ffi  \right\|_X. \nonumber
\end{align*} 
\end{lemma}

\begin{proof}
Denote $B:= F^{m(t)-1}(s/m(t))+F^{m(t)-2}(s/m(t))+\ldots+F(s/m(t)) +\id$. Then 
$F^{m(t)}(s/m(t))\ffi - \ffi=B(F(s/m(t)) -\id) \ffi$ and $\| B\|\leq m(t)$. Therefore,  one has
\begin{align*}
&\frac{\Psi_t(s)}{s}=\left\| \frac{F^{m(t)}(s/m(t))\ffi- T_s \ffi}{s}  \right\|_X\leq\\
&
\quad\leq\left\| \frac{F^{m(t)}(s/m(t))\ffi-  \ffi}{s} -L\ffi \right\|_X + \left\| \frac{T_s\ffi - \ffi}{s} - L\ffi   \right\|_X
\end{align*}
and
\begin{align*}
&\left\| \frac{F^{m(t)}(s/m(t))\ffi-  \ffi}{s} -L\ffi \right\|_X=\\
&
=\left\| \frac{(m^{-1}(t)B)(F(s/m(t))\ffi-  \ffi)}{s/m(t)} -(m^{-1}(t)B)L\ffi +(m^{-1}(t)B)L\ffi -L\ffi \right\|_X\leq\\
&
\leq \left\| \frac{F(s/m(t))\ffi-  \ffi}{s} -L\ffi \right\|_X + \left\| (m^{-1}(t)B)L\ffi -L\ffi \right\|_X.
\end{align*}
\end{proof}

\begin{lemma}\label{le:estimate of Psi-2}
Let  $\Psi_t$ be as in Lemma \ref{le:estimate of Psi}. For each $\eps>0$ there exist $t_\eps>0$ and $s_\eps>0$ such that for all $t\in(0,t_\eps]$ and all $s\in(0,s_\eps]$ holds the estimate
$$
\frac{\Psi_t(s)}{s}<\eps.
$$
\end{lemma}

\begin{proof}
Fix $\eps>0$. Choose $s_1>0$ such that $\left\| \frac{T_s\ffi - \ffi}{s} - L\ffi   \right\|_X <\eps/3$ for all $s\in(0,s_1]$. This can be done since $\ffi\in D\subset \Dom(L)$.  Choose then $t_1>0$ such that for all $s\in(0,s_1]$ one has  $ \left\| \frac{F(s/m(t_1))\ffi - \ffi}{s/m(t_1)} - L\ffi   \right\|_X<\eps/3$. This can be done due to assumption $\lim_{t\to0}\big\|\frac{F(t)\ffi-\ffi}{t}-L\ffi\big\|_X=0$ for each $\ffi\in D$.  Since $s/m(t)\leq s_1/m(t_1)$ for all $s\in(0,s_1]$ and $t\in(0,t_1]$, one has also $ \left\| \frac{F(s/m(t))\ffi - \ffi}{s/m(t)} - L\ffi   \right\|_X<\eps/3$ for such $s$ and $t$. Since the semigroup  $(T_t)_{t\geq0}$ is strongly continuous choose   $s_2\in(0,s_1]$ such that  $\| T_\tau L\ffi-L\ffi \|_X<\eps/9$ for all $\tau\in (0,s_2]$. Due to the Chernoff Theorem it is possible to choose $K\in\NN$ such that  for all $k\geq K$ and all $\tau\in [0,s_2/m(t_1)]$ the inequality $\left\|  F^{k-1}(\tau/(k-1))L\ffi - T_\tau L\ffi \right\|_X<\eps/9$  holds.  Choose $t_2\in (0,t_1]$ such that $m(t_2)> K$.
Thus,  the following estimate is true for  $s\in (0,s_2]$ and $t\in(0,t_2]$
\begin{align*}
&\left\| \frac{1}{m(t)}\suml_{k=1}^{m(t)} F^{k-1}(s/m(t)) L\ffi - L\ffi\right\|_X\leq\\
&
\leq \frac{1}{m(t)} \suml_{k=1}^{m(t)} \left\|F^{k-1}\left(\frac{(k-1)s/m(t)}{k-1}\right) L\ffi -  T_{(k-1)s/m(t)} L\ffi\right\|_X +\\
&
\quad\quad\quad+\frac{1}{m(t)}\suml_{k=1}^{m(t)}\left\|  T_{(k-1)s/m(t)} L\ffi-L\ffi\right\|_X\leq\\
&
\leq \frac{1}{m(t)} \suml_{k=K}^{m(t)} \left\|F^{k-1}\left(\frac{(k-1)s/m(t)}{k-1}\right) L\ffi -  T_{(k-1)s/m(t)} L\ffi\right\|_X + \frac{2K\|L\ffi\|_X}{ m(t)} + \eps/9\\
&
\leq 2K\|L\ffi\|_X m^{-1}(t) +2\eps/9.
\end{align*}
Due to our assumptions the function $m:(0,\infty)\to\mathds{N}_0$ is  monotone  with $m(t)\to\infty$ as $t\to0$. Therefore, one can choose  $t_3\in (0,t_2]$ with  $m(t_3)> \frac{18K\|L\ffi\|_X}{\eps}$. Then    due to Lemma \ref{le:estimate of Psi} with $t_\eps:=t_3$  and $s_\eps:=s_2$  holds
$$
\frac{\Psi_t(s)}{s}<\eps.
$$
\end{proof}

\begin{lemma}\label{le:5}
For each  $\ffi\in D$  holds:
$$
\lim_{t\to0}\left\|\frac{\mathcal{F}_0(t)\ffi-\ffi}{t}-L^{f_0}\ffi\right\|_X=0.
$$
\end{lemma}
\begin{proof} With the function $\Psi_t$  defined in Lemma \ref{le:estimate of Psi}, one has

\begin{align*}
&\lim_{t\to0}\left\|\frac{\mathcal{F}_0(t)\ffi-\ffi}{t}-L^{f_0}\ffi\right\|_X
\leq \lim_{t\to0}\left\|\frac{\mathcal{F}_0(t)\ffi-T^{f_0}\ffi}{t}\right\|_X + \lim_{t\to0}\left\|\frac{T^{f_0}\ffi-\ffi}{t}-L^{f_0}\ffi\right\|_X=\\
&
= \lim_{t\to0}\frac1t\left\|\intl_0^\infty(F^{m(t)}(s/m(t))\ffi- T_s \ffi )\eta^0_t(ds)\right\|_X
\leq \lim_{t\to0}\frac1t \intl_0^\infty \Psi_t(s)\eta^0_t(ds).
\end{align*}
Fix an arbitrary $\eps>0$. Take $t_\eps$ and $s_\eps$ as in Lemma \ref{le:estimate of Psi-2}. Let $r_\eps:=\min(s_\eps, 1)$. Take $R_\eps>0$ such that $\int_{R_\eps}^\infty \mu(ds)<\eps$. For $k=1,2,3$ choose functions $\chi_k\in C^2_b(\cR)$ with $0\leq \chi_k\leq1$ such that  $\supp \chi_1\subset (-1, r_\eps)$, $\supp \chi_3\subset (R_\eps,\infty)$, $\supp \chi_2\subset (r_\eps/2, 2R_\eps)$ and $\sum_{k=1}^3\chi_k(s)=1$ for all $s\geq 0$. Then by Lemma \ref{le:estimate of Psi-2}
\begin{align*}
&\lim_{t\to0}\frac1t \intl_0^\infty \Psi_t(s)\eta^0_t(ds)\leq\lim_{t_\eps>t\to0}\frac{\eps}{t} \intl_0^{r_\eps} s\chi_1(s)\eta^0_t(ds)+ \\
&
+\lim_{t\to0}\sup\limits_{s\in[r_\eps/2, 2R_\eps]}\Psi_t(s)\cdot \frac1t \intl_{r_\eps/2}^{2R_\eps}\chi_2(s)\eta^0_t(ds) + \lim_{t\to0}\frac{2\|\ffi\|_X}{t} \intl_{R_\eps}^\infty \chi_3(s)\eta^0_t(ds).
\end{align*}
Due to the Chernoff theorem $\lim_{t\to0}\sup\limits_{s\in[r_\eps/2, 2R_\eps]}\Psi_t(s)=0$ for any fixed $r_\eps$ and $R_\eps$.  Define also $\chi_4$ such that  $\chi_4(s):=s\chi_1(s)$ for all $s\in\cR$. Since the semigroup $(\bar {S}^{\eta^0}_t)_{t\geq 0} $ is strong Feller, $\chi_k\in C^2_b(\cR)\subset \Dom(\bar {L}^{\eta^0})$ and $\chi_k(0)=0$ for $k=2,3,4$, one has  
\begin{align*}
\lim_{t\to0}\frac1t\intl_{0+}^\infty \chi_k(s)\eta_t^0(ds)=\lim_{t\to0}\frac{\bar {S}^{\eta^0}_t\chi_k - \chi_k}{t}(0)=\left(\bar {L}^\eta \chi_k\right)(0)=\intl_{0+}^\infty \chi_k(s)\mu(ds).
\end{align*}
Therefore, $\intl_{0+}^\infty \chi_2(s)\mu(ds)=\intl_{r_\eps/2}^{2R_\eps}\chi_2(s)\mu(ds)\leq \mu[r_\eps/2,2R_\eps]<\infty$ (cf. Lemma 2.16 of \cite{MR3156646}).  And hence
$$
\lim_{t\to0}\sup\limits_{s\in[r_\eps/2, 2R_\eps]}\Psi_t(s)\cdot \frac1t \intl_{r_\eps/2}^{2R_\eps}\chi_2(s)\eta^0_t(ds)=0.
$$
Similarly, 
\begin{align*}
\liml_{t\to0}\frac{2\|\ffi\|_X}{t} \intl_{R_\eps}^\infty \chi_3(s)\eta^0_t(ds)= 2\|\ffi\|_X\intl_{R_\eps}^\infty \chi_3(s)\mu(ds)< 2\eps\|\ffi\|_X.
\end{align*}
And, further,  with $K:=\int_0^{1}s\mu(ds)<\infty$
$$
\lim_{t_\eps>t\to0}\frac{\eps}{t} \intl_0^{r_\eps} s\chi_1(s)\eta^0_t(ds)=\eps\intl_0^{r_\eps}s\chi_1(s)\mu(ds)\leq \eps\intl_0^{1}s\mu(ds)=K\eps.
$$ 
 Thus, it is shown that for each fixed $\eps>0$
 $$
 \lim_{t\to0}\left\|\frac{\mathcal{F}_0(t)\ffi-\ffi}{t}-L^{f_0}\ffi\right\|_X\leq \eps(K+2\|\ffi\|_X).
 $$
 Therefore, the statement of Lemma is true.
\end{proof}

Hence the family $(\mathcal{F}_0(t))_{t\geq0}$ is Chernoff equivalent to the semigroup $(T^{f_0}_t)_{t\geq 0}$ subordinate   to $(T_t)_{t\geq 0}$ with respect to  $(\eta^0_t)_{t\geq 0}$. And Theorem \ref{thm:ChFam_knownSub-r} is proved.

\end{proof} %%% of the theorem

\begin{remark}\label{rem:variable b}
Let now  $X=C_b(Q)$ or $X=C_\infty(Q)$, where $Q$  is a  metric space. Let $\sigma$, $\lambda$ be not  constants  but  continuous functions on $Q$ such that $\lambda$ is bounded and strictly positive and $\sigma$ is bounded from below. Assume that the operator  $L^f$ defined as in \eqref{eq:L^f} (but with variable $\sigma$ and $\lambda$) with the domain $D$ (here  $D$ is as in Theorem \ref{thm:ChFam_knownSub-r}) is closable and the closure generates a strongly continuous semigroup $(T^f_t)_{t\geq0}$ on $X$. Then due to Theorem \ref{TH5.1}, Theorem \ref{thm:Technomag mult pert} and  Lemmas \ref{le:1} -- \ref{le:5}  the family $(\widetilde{\mathcal{F}}(t))_{t\geq 0}$  of operators on $(X,\|\cdot\|_\infty)$ defined by $\widetilde{\mathcal{F}}(0)=\id$ and
\begin{equation}\label{eq:mathcal(F)}
\widetilde{\mathcal{F}}(t)\ffi=  e^{-\sigma t}\circ   \widetilde{F}(t)\circ \mathcal{F}_0(t)\ffi,\quad t>0,\, \ffi\in X,
\end{equation}
with $(\mathcal{F}_0(t))_{t\geq0}$ as in Theorem \ref{thm:ChFam_knownSub-r} and with 
$(\widetilde{F}(t))_{t\geq0}$ such that
\begin{equation}\label{eq:widetilde F}
\widetilde{F}(t)\ffi(x):=\left(F(\lambda(x)t)\ffi   \right)(x), \quad \forall\ffi\in X,\quad\forall x\in Q,
\end{equation}
is Chernoff equivalent to  the semigroup $(T^f_t)_{t\geq0}$.
\end{remark}

\subsection{Case 2:  L\'{e}vy measures of subordinators are known and bounded}\label{sec:knownLevy}
In this 
subsection we again consider the semigroup $(T^f_t)_{t\geq 0}$ subordinate   to a given semigroup $(T_t)_{t\geq 0}$ with respect to a given convolution semigroup $(\eta_t)_{t\geq 0}$ associated to a Bernstein function $f$ defined by a triplet $(\sigma,\lambda,\mu)$.  We assume that the corresponding convolution semigroup $(\eta^0_t)_{t\geq 0}$ associated to the Bernstein function $f_0$ defined by the triplet $(0,0,\mu)$ is not known explicitly. In this case  the family $(\mathcal{F}_0(t))_{t\geq0}$ of Theorem \ref{thm:ChFam_knownSub-r} is not known explicitly as well, and hence the formula \eqref{eq:mathcal(F)} is not proper  for direct computations any more. Let us assume that  the L\'{e}vy measure $\mu$ of $(\eta^0_t)_{t\geq 0}$ is given explicitly and is bounded (and nonzero). In this case the generator $L^{\eta^0}$  of the corresponding semigroup  $S^{\eta^0}_t$ is a bounded linear operator given as in \eqref{eq: L^eta}. The generator $L^{f_0}$ of the semigroup  $(T^{f_0}_t)_{t\geq 0}$ subordinate   to $(T_t)_{t\geq 0}$ with respect to  $(\eta^0_t)_{t\geq 0}$ is given by \eqref{eq:L^f0} and is also bounded.  Therefore, the semigroup $(T^{f_0}_t)_{t\geq 0}$ can be constructed, e.g., via  Taylor series representation. We use another approach. 

\begin{theorem}\label{thm:ChFam_known bdd Levy}
Let $(T_t)_{t\geq 0}$  be a strongly continuous contraction semigroup on a Banach space $(X,\|\cdot\|_X)$ with the generator $(L,\D(L))$. Let  $(F(t))_{t\geq 0}$ be a family of contraction operators on $(X,\|\cdot\|_X)$ which is Chernoff equivalent to $(T_t)_{t\geq 0}$, i.e. $F(0)=\id$,  $\|F(t)\|\le 1$ for all  $t\geq 0$ and there is a set $D\subset \D(L)$, which is a core for $L$, such that  $\lim_{t\to0}\big\|\frac{F(t)\ffi-\ffi}{t}-L\ffi\big\|_X=0$ for each $\ffi\in D$. Let $f$ be a Bernstein function  given by a triplet $(\sigma,\lambda,\mu)$ through  the representation \eqref{eq:BernsteinFunc}  with associated convolution semigroup $(\eta_t)_{t\geq 0}$ supported by $[0,\infty)$.
 %%Let $(\eta^0_t)_{t\geq 0}$ be the (supported by $[0,\infty)$) convolution semigroup associated to the Bernstein      function $f_0$ defined by the triplet $(0,0,\mu)$. 
Assume that the measure $\mu$ is bounded. 
 Let $m:(0,\infty)\to\mathds{N}_0$ be a monotone function with $m(t)\to\infty$ as $t\to0$. %, e.g., $m(t):=\left[ \frac1t  \right]$. 
 Let $(T^f_t)_{t\geq 0}$ be the semigroup subordinate   to $(T_t)_{t\geq 0}$ with respect to  $(\eta_t)_{t\geq 0}$ and $L^f$ be  its generator. Consider a family   $(\mathcal{F}_\mu(t))_{t\geq 0}$  of operators on $(X,\|\cdot\|_X)$ defined for all $\ffi\in X$ and all $t\geq0$ by 
\begin{equation}\label{eq:mathcal(F)mu}
\mathcal{F}_\mu(t)\ffi=  e^{-\sigma t} F(\lambda t)\left( \ffi+ t \intl_{0+}^\infty (F^{m(t)}(s/m(t))\ffi-\ffi)\mu(ds)   \right).
\end{equation}
The family $(\mathcal{F}_\mu(t))_{t\geq 0}$ is Chernoff equivalent to the semigroup $(T^f_t)_{t\geq 0}$ and, hence
%%, the Feynman formula
$$
T^f_t \ffi=\lim_{n\to\infty} \big[\mathcal{F}_\mu(t/n)   \big]^n \ffi
$$
for all $\ffi\in X$ locally uniformly with respect to $t\geq  0$.
\end{theorem}

\begin{proof}
Let us prove that the family $(F_\mu(t))_{t\geq0}$, defined for all $\ffi\in X$ and all $t\geq0$ by
\begin{equation}\label{eq:F mu}
F_\mu(t)\ffi:= \ffi+ t \intl_{0+}^\infty (F^{m(t)}(s/m(t))\ffi-\ffi)\mu(ds),
\end{equation}
is Chernoff equivalent to the semigroup $(T^{f_0}_t)_{t\geq 0}$ which is subordinate   to $(T_t)_{t\geq 0}$ with respect to  $(\eta^0_t)_{t\geq 0}$ associated to the Bernstein function $f_0$ defined by the triplet $(0,0,\mu)$. Then  the statement of Theorem \ref{thm:ChFam_known bdd Levy} follows immediately from Theorem \ref{TH5.1}. So, let $K:=\mu(\cR)<\infty$. Then,   clearly,
$F_\mu(0)=\id$,  
\begin{align*}
\|F_\mu(t)\ffi\|_X\leq \|\ffi\|_X+tK\left\| F^{m(t)}(s/m(t))\ffi-\ffi \right\|_X\leq \|\ffi\|_X(1+2tK)\leq e^{2tK}\|\ffi\|_X 
\end{align*} 
 \begin{align*}
\text{and }\quad \|F_\mu(t)\ffi-\ffi\|_X\leq tK\left\| F^{m(t)}(s/m(t))\ffi-\ffi \right\|_X\leq 2tK\|\ffi\|_X\to0,\quad t\to0.
 \end{align*}
 Further, for an arbitrary  $\eps>0$  choose $R_\eps$ such that $\int_{R_\eps}^\infty\mu(ds)<\eps$. Then for each $\ffi\in D$
 \begin{align*}
&\liml_{t\to0}\left\|  \frac{F_\mu(t)\ffi-\ffi}{t} -L^{f_0}\ffi \right\|_X=\liml_{t\to0}\left\| \intl_{0+}^\infty (F^{m(t)}(s/m(t))\ffi-T_s\ffi)\mu(ds) \right\|_X \leq\\
&
\leq \liml_{t\to0}\left[\intl_{0+}^{R_\eps} \|F^{m(t)}(s/m(t))\ffi-T_s\ffi\|_X\mu(ds)+\intl_{R_\eps}^\infty \|F^{m(t)}(s/m(t))\ffi-T_s\ffi\|_X\mu(ds)\right]\leq \\
&
2\|\ffi\|_X \eps+K\liml_{t\to0}\sup\limits_{s\in[0,R_\eps]}\|F^{m(t)}(s/m(t))\ffi-T_s\ffi\|_X = 2\|\ffi\|_X \eps
 \end{align*}
 due to the Chernoff Theorem. Therefore, 
 $$
 \liml_{t\to0}\left\|  \frac{F_\mu(t)\ffi-\ffi}{t} -L^{f_0}\ffi \right\|_X=0
 $$
 which proves the statement of Theorem \ref{thm:ChFam_known bdd Levy}.
\end{proof}

\begin{remark}
 The choise of $F_\mu(t)$ is motivated by the fact that  for each bounded linear operator $A$ the family $F_A(t):=\id+tA$  is obviously Chernoff equivalent to the semigroup $e^{tA}$. We have, however, the family $F_A(t):=\id+tA(t)$, where operators $A(t)$ are bounded and tend to the generator $A$ as $t\to0$.
The natural question arises:  if it is possible to find the family $F_A(t):=\id+tA(t)$, where operators $A(t)$ are bounded and tend to the unbounded generator $A$ of the semigroup $e^{tA}$ as $t\to0$, such that $F_A(t)$ would be Chernoff equivalent to $e^{tA}$? In this case it would be possible   to generalize  Theorem \ref{thm:ChFam_known bdd Levy} to the case of unbounded L\'{e}vy measure $\mu$, e.g., by approximating $\mu$ with  bounded measures $\mu_t:=1_{[\alpha(t),\infty)}\mu$ for some proper function $\alpha(t)\to0$ as $t\to0$.
However, the answer is NO, since the required in the Chernoff Theorem norm estimate $\|F_A(t)\|\leq e^{ct}$  for all $t\geq0$ and some $c\in\cR$ (or the equivalent one $\|F(A(t)\|^k\leq Me^{ckt}$ for all $k\in\NN$,  $t\geq0$ and some $c\in\cR$, $M\geq 1$, cf. \cite{MR710486})  fails.
\end{remark}

\begin{remark}
The analogue of Remark \ref{rem:variable b} is true also for the family $(\widetilde{\mathcal{F}}_\mu(t))_{t\geq 0}$, 
\begin{equation*}
\widetilde{\mathcal{F}}_\mu(t)\ffi:=  e^{-\sigma t} \widetilde{F}(t)\left( \ffi+ t \intl_{0+}^\infty (F^{m(t)}(s/m(t))\ffi-\ffi)\mu(ds)   \right), 
\end{equation*} 
with $\widetilde{F}(t)$ as in \eqref{eq:widetilde F}.
\end{remark}

\section{Applications}
\subsection{Approximation of subordinate   Feller semigroups}

A \emph{Feller process} $(X_t)_{t\geq0}$ with a state space $\cR^d$ is a strong Markov process whose associated operator semigroup $(T_t)_{t\geq0}$, $T_t\ffi(x):=\mathbb{E}^x\left[ \ffi(X_t)  \right]$, $\ffi\in C_\infty(\cR^d)$, $t\geq0$, $x\in\cR^d$, is a strongly continuous positivity preserving\footnote{A semigroup $(T_t)_{t\geq0}$ on $C_\infty(\cR^d)$ is positivity preserving if $T_t\ffi\geq0 $ for all $\ffi\in C_\infty(\cR^d)$ with $\ffi\geq0$ and all $t>0$. } contraction semigroup on $C_\infty(\cR^d)$. The semigroup $(T_t)_{t\geq0}$ is said to be a \emph{Feller semigroup}.  Due to results of P. Courr\`{e}ge \cite{MR0245085}, \cite{Courrege}, and W. von Waldenfels \cite{vWaldenfels},  \cite{MR0166603}, \cite{MR0193678}, if the set $C^\infty_c(\cR^d)$ of all infinitely differentiable functions with compact support belongs to the domain $\Dom(L)$ of the generator $L$ of a Feller semigroup $(T_t)_{t\geq0}$, then the restriction of $L$ onto $C^\infty_c(\cR^d)$ is a pseudo-differential operator (PDO for short) with symbol $-H$, i.e.
$$
L\ffi(x):=-(2\pi)^{-d}\intl_{\cR^d}\intl_{\cR^d}e^{ip\cdot (x-q)}H(x,p)\ffi(q)dqdp,
$$
the function $H:\cR^d\times\cR^d\to\cC$ is measurable, locally bounded in both variables $(q,p)$ and for each fixed $q$ satisfies the L\'{e}vy-Khintchine representation
\begin{align}\label{eq:Levy-Khintchine}
H(q,p)=C(q)+iB(q)\cdot p +\frac12 \,p\cdot A(q)p+\intl_{\cR^d\setminus \{0\}}\left( 1-e^{iy\cdot p}+\frac{iy\cdot p}{1+|y|^2}   \right)\nu(q,dy),
\end{align}
where, for each fixed $q$,  $C(q)\geq0$, $B(q)\in\cR^d$, $A(q)$ is a   positive semidefinite symmetric matrix  and   $\nu(q,\cdot)$ is a positive Radon measure on $\cR^d\setminus \{0\}$ satisfying $\int_{\cR^d\setminus \{0\}}\min(1,|y|^2)\nu(q,dy)<\infty$. Each operator $T_t$ can itself be represented as a PDO with the symbol $\Lambda_t(q,p):=\mathbb{E}^q\left[e^{ip\cdot(X_t-q) } \right]$. If $(X_t)_{t\geq0}$ is a L\'{e}vy process, we have $H(q,p)=H(p)$ and $\Lambda_t(q,p)=e^{-tH(p)}$.  In general, however, $\Lambda_t(q,p)\neq e^{-tH(q,p)}$. Nevertheless, the family $(F(t))_{t\geq0}$ of PDOs with symbol $e^{-tH(q,p)}$ are Chernoff equivalent to $(T_t)_{t\geq0}$. Namely, by Theorem 3.1 and  Remark 3.2 %and Lemma 3.3
 of \cite{MR2999096}, the following statement is true.
\begin{proposition}\label{prop:BShcS}
(i) Let the function $H:\cR^d\times\cR^d\to\cC$ be measurable and locally bounded in both variables, $H(q,\cdot)$ satisfy for each fixed $q$ the L\'{e}vy-Khintchine representation \eqref{eq:Levy-Khintchine}. Assume that
\begin{align*}\label{a}
&\displaystyle\sup_{q\in\rn} |H(q,p)| \leq
                 \kappa(1+|p|^2)\quad \text{for all}\quad p\in\rn \quad\text{and some}\quad \kappa>0 ,\\
&
  \displaystyle p\mapsto H(q,p)\quad \text{is uniformly (w.r.t.} \quad q\in\rn \text{)
                 continuous at}\quad p=0,\\
&
\displaystyle q\mapsto H(q,p)\quad \text{is continuous for
all}\quad p\in\rn.
\end{align*}
Assume also that  the PDO with symbol $-H$ defined on $C^\infty_c(\cR^d)$ is closable in $C_\infty(\cR^d)$  and the closure generates a Feller semigroup $(T_t)_{t\geq0}$. Consider  the family  $(F(t))_{t\geq0}$,
$$
F(t)\ffi(x):=(2\pi)^{-d}\intl_{\cR^d}\intl_{\cR^d} e^{ip\cdot(x-q)}e^{-tH(x,p)}\ffi(q)dqdp,\quad \forall\,\ffi\in C^\infty_c(\cR^d).
$$
Then the operators $F(t)$ can be extended to  contractions on  $C_\infty(\cR^d)$ and the family $(F(t))_{t\geq0}$ is Chernoff equivalent to the semigroup $(T_t)_{t\geq0}$, i.e., for each $\ffi\in C_\infty(\cR^d)$ the Chernoff approximation   $T_t \ffi= \lim_{n\to\infty}\left[F(t/n)\right]^n\ffi$ holds.  

\noindent (ii) Assume additionally that the function $H$  satisfies the following condition:
\begin{equation*}
\exists\, C>0\quad \,\mbox{such that }\,
\big\|\partial^\alpha_q\partial^\beta_p
e^{tH}\big\|_{L^\infty(\rn\times\rn)}\le C,
\end{equation*}
where $ \alpha,\,\beta\in \mathds{N}_0^d$,  $\, \alpha=0\,\mbox{ or
} \,1$, $  \beta=0\,\mbox{ or } \,1$,
$\,\partial^\alpha_q\partial^\beta_p$ are derivatives in  the distributional
sense.  Then  the operators 
 $F(t)$ extend to bounded linear operators on $ L^2(\real^d)$.
 And the Chernoff approximation turns for each $\ffi\in C_\infty(\cR^d)\cap L^2(\cR^d)$ into  the following Feynman formula (with $q_{n+1}:=q$): 
\begin{align*}
&(T_t\ffi)(q)\\
\notag&=\lim\limits_{n\rightarrow\infty}(2\pi)^{-dn}\int\limits_{(\real^d)^{2n}}
e^{i\sum\limits_{k=1}^n p_k\cdot(q_{k+1}-q_k)}e^{-\frac{t}{n}\sum\limits_{k=1}^n
H(q_{k+1},p_k)}\ffi(q_1)dq_1dp_1\cdots dq_ndp_n,
\end{align*}
where the equality holds  in $L^2$-sense and  the integrals in the right hand side must be considered in a regularized sense. 
%% We refer to \cite{Hwang87} for further  conditions on $H(q,p)$ ensuring   $F(t): L_2(\real^d)\to L_2(\real^d)$.
\end{proposition}
Let now the semigroup $(T^f_t)_{t\geq0}$ be subordinate   to the semigroup $(T_t)_{t\geq 0}$ with respect to a given convolution semigroup $(\eta_t)_{t\geq 0}$ associated to a Bernstein function $f$ defined by a triplet $(\sigma,\lambda,\mu)$. The statement below follows immediately from Theorem \ref{thm:ChFam_knownSub-r}, Theorem \ref{thm:ChFam_known bdd Levy} and Proposition \ref{prop:BShcS}.
\begin{theorem}
(i) Under all assumptions and notations of Proposition \ref{prop:BShcS} and Theorem \ref{thm:ChFam_knownSub-r}  the following Feynman formula is true for each $\ffi\in C_\infty(\cR^d)\cap L^2(\cR^d)$ with $N:=1+m(t/n)$ and $q_{1+nN}:=q$
\begin{align*}
&(T^f_t\ffi)(q)=\lim\limits_{n\rightarrow\infty}(2\pi)^{-2dnN}e^{-\sigma t}\int\limits_{(0,\infty)^{n}}\int\limits_{\real^{2dnN}}\exp\left({-\frac{\lambda t}{n}\suml_{j=1}^n H(q_{1+jN},p_{jN})}\right)\times\\
\notag&
\times  \exp\left({i \suml_{j=1}^n \suml_{k=1}^{N} p_{k+(j-1)N}\cdot(q_{k+1+(j-1)N}-q_{k+(j-1)N}) }\right)\times\\
&
\times \exp\left({-\suml_{j=1}^n \frac{s_j}{N-1} \suml_{k=1}^{N-1} H(q_{k+1+(j-1)N}, p_{k+(j-1)N}) }\right)\ffi(q_1)\times\\
&
\times\prodl_{j=1}^n \prodl_{k=1}^{N-1}dq_{k+(j-1)N}dp_{k+(j-1)N}\eta^0_{t/n}(ds_j). 
\end{align*}

\noindent (ii) Under all assumptions and notations of Proposition \ref{prop:BShcS} and Theorem \ref{thm:ChFam_known bdd Levy}  the following Feynman formula is true for each $\ffi\in C_\infty(\cR^d)\cap L^2(\cR^d)$ with $N:=1+m(t/n)$ and $K:=(\mu(0,\infty))^{-1}<\infty$
\begin{align*}
&(T^f_t\ffi)(q)=\lim\limits_{n\rightarrow\infty}(2\pi)^{-2dnN}e^{-\sigma t}\int\limits_{(0,\infty)^{n}}\int\limits_{\real^{2dnN}}\exp\left({-\frac{\lambda t}{n}\suml_{j=1}^n H(q_{1+jN},p_{jN})}\right)\times\\
\notag&
\times \exp\left({i \suml_{j=1}^n \suml_{k=1}^{N} p_{k+(j-1)N}\cdot(q_{k+1+(j-1)N}-q_{k+(j-1)N}) }\right)\times\\
&
\times \prodl_{j=1}^n\left[(K-t)+   t\exp\left({- \frac{s_j}{N-1} \suml_{k=1}^{N-1} H(q_{k+1+(j-1)N}, p_{k+(j-1)N}}) \right)\right]\ffi(q_1)\times\\
&
\times\prodl_{j=1}^n \prodl_{k=1}^{N-1}dq_{k+(j-1)N}dp_{k+(j-1)N}\mu(ds_j), 
\end{align*}
The identities in both Feynman formulae hold  in $L^2$-sense and  the integrals in the right hand sides must be considered in a regularized sense. 
\end{theorem}

\subsection{Approximation of subordinate   diffusions in $\cR^d$}
%%%+multiplicative perturbations
In this subsection we consider the case of Feller diffusion, i.e. we assume that the measure $\nu$ in the Levy--Khintchine representation \eqref{eq:Levy-Khintchine} of the symbol $H$ of generator $(L,\Dom(L))$ is identically zero. In this case the PDO with symbol $-H$ is just a second order differential operator with variable coefficients and the following results are true (see  \cite{Yana_Technomag14}, cf.  \cite{MR2729591}).
\begin{proposition}\label{prop:BGS}
Let $A\in C(\cR^d,\mathcal{L}(\cR^d))$ be such that $A(x)$ is a positive definite symmetric matrix for each $x\in\cR^d$, let $B\in C(\cR^d,\cR^d)$ and let $C\in C(\cR^d,\cR)$ be nonnegative. Consider the operator $L$ defined for all $\ffi\in C^2(\cR^d)$ by the formula
\begin{equation}
L\ffi(x):= \frac12\tr(A(x)\Hess\ffi(x))+B(x)\cdot \nabla\ffi(x)-C(x)\ffi(x).
\end{equation}
Assume that there exists $\alpha\in(0,1]$ such thath the closure of $(L, C^{2,\alpha}_c(\cR^d))$ generates a strongly continuous semigroup $(T_t)_{t\geq 0}$ on the space $C_\infty(\cR^d)$. Consider the family  $(F(t))_{t\geq 0}$ such that for each  $\varphi\in C_\infty(\mathbb{R}^d)$ 
\begin{align}\label{eq:F(t) BGS}
&F(t)\varphi(x)=\\
&
=\frac{\exp(-tC(x))}{\sqrt{\det A(x)(2\pi t)^{d}}}  \intl_{\mathbb{R}^d} \exp\bigg(\frac{-A^{-1}(x)(x-y+tB(x))\cdot(x-y+tB(x))}{2t}\bigg){\varphi}(y)dy,\nonumber
\end{align}
Then the family  $(F(t))_{t\geq 0}$ is Chernoff equivalent to the semigroup  $(T_t)_{t\geq  0}$  and the Chernoff approximation
 \begin{equation*}
T_t \ffi= \lim_{n\to\infty}\big(F(t/n)\big)^n \ffi
\end{equation*}
holds for all $\ffi\in C_\infty(\cR^d)$ locally uniform with respect to   $t \geq0$.  Therefore, the following Feynman formula is true for each  $\varphi\in C_\infty(\mathbb{R}^d)$  with $q_{n+1}:=q$
\begin{align*}
T_t\varphi(q)=&\liml_{n\to\infty} \intl_{\mathbb{R}^{dn}} \exp\bigg(-\suml_{k=1}^n
A^{-1}(q_{k+1})B(q_{k+1})\cdot(q_{k+1}-q_{k})\bigg)\times\\
&
\times\exp\bigg(-\frac{t}{n}\suml_{k=1}^n C(q_{k+1})\bigg)\exp\bigg(-\frac{t}{2n}\suml_{k=1}^n A^{-1}(q_{k+1})B(q_{k+1})\cdot B(q_{k+1}) \bigg)\times\nonumber\\
&
\times \varphi(q_1)\prodl_{k=1}^n p_A(t/n,q_k,q_{k+1}) dq_1 \ldots dq_n,\nonumber
\end{align*}
where for each  $ x,y \in {\mathbb R}^d$, $t > 0$
\begin{equation}\label{eq:p_A}
p_A(t,x,y) :=\frac{1}{\sqrt{\det A(x)(2\pi t)^{d}}} \exp\bigg(-\frac{A^{-1}(x)(x-y)\cdot(x-y)}{2t}\bigg).
\end{equation}
\end{proposition}
Let now the semigroup $(T^f_t)_{t\geq 0}$ be subordinate   to the semigroup $(T_t)_{t\geq 0}$ with respect to a given convolution semigroup $(\eta_t)_{t\geq 0}$ associated to a Bernstein function $f$ defined by a triplet $(\sigma,\lambda,\mu)$. The statement below follows immediately from Theorem \ref{thm:ChFam_knownSub-r}, Theorem \ref{thm:ChFam_known bdd Levy} and Proposition \ref{prop:BGS}.
\begin{theorem}
(i) Under  assumptions and notations of Proposition \ref{prop:BGS} and Theorem \ref{thm:ChFam_knownSub-r} the  family $(\mathcal{F}(t))_{t\geq0}$ of Theorem \ref{thm:ChFam_knownSub-r}, constructed with the help of the family $(F(t))_{t\geq0}$ given in \eqref{eq:F(t) BGS},  is Chernoff equivalent to the semigroup $(T^f_t)_{t\geq 0}$ and has the following explicit view (with $p_A$ as in \eqref{eq:p_A}):

\begin{align*}
&\mathcal{F}(0):=\id\quad\text{ and  with }\quad q_{m(t)+2}:=q \\
&
\mathcal{F}(t)\varphi(q):= e^{-t\sigma}\intl_{0+}^\infty\intl_{\mathbb{R}^{d(m(t)+1)}}\exp\left(-\left[t\lambda C(q_{m(t)+2})+\frac{s}{m(t)}\suml_{k=1}^{m(t)} C(q_{k+1})\right]\right)\times\\
&
\times\exp\left(-\suml_{k=1}^{m(t)+1}
A^{-1}(q_{k+1})B(q_{k+1})\cdot(q_{k+1}-q_{k})\right)\times\\
&
\times \exp\left(-\frac{s}{m(t)}\suml_{k=1}^{m(t)} A^{-1}(q_{k+1})B(q_{k+1})\cdot B(q_{k+1}) \right)\times\\
&
\times \exp\left(-\frac{t\lambda}{2} A^{-1}(q_{m(t)+2})B(q_{m(t)+2})\cdot B(q_{m(t)+2})\right)         \varphi(q_1)\times\\
&
\times\left[ p_A(t\lambda,q_{m(t)+1},q_{m(t)+2})\prodl_{k=1}^{m(t)}p_A(s/m(t),q_{k},q_{k+1}) \right]
 \prodl_{k=1}^{m(t)}dq_{k}\eta^0_{t}(ds).
\end{align*}
Therefore, 
for each $\ffi\in C_\infty(\cR^d)$ the following Feynman formula is true  with  $N:=1+m(t/n)$ and $q_{1+nN}:=q$
\begin{align*}
&T^f_t\varphi(q)=\\
&
=\liml_{n\to\infty} e^{-t\sigma}\intl_{(0,\infty)^n}\,\,\,\intl_{\mathbb{R}^{dnN}}\exp\left(-\suml_{j=1}^n\left[\frac{t\lambda}{n} C(q_{1+jN})+\frac{s_j}{N-1}\suml_{k=1}^{N-1} C(q_{k+1+(j-1)N})\right]\right)\times\\
&
\times\exp\left(-\suml_{j=1}^n\suml_{k=1}^N
A^{-1}(q_{k+1+(j-1)N})B(q_{k+1+(j-1)N})\cdot(q_{k+1+(j-1)N}-q_{k+(j-1)N})\right)\times\\
&
\times \exp\left(-\suml_{j=1}^n\frac{s_j}{N-1}\suml_{k=1}^{N-1} A^{-1}(q_{k+1+(j-1)N})B(q_{k+1+(j-1)N})\cdot B(q_{k+1+(j-1)N}) \right)\times\\
&
\times \exp\left(-\suml_{j=1}^n\frac{t\lambda}{2n} A^{-1}(q_{1+jN})B(q_{1+jN})\cdot B(q_{1+jN})\right)         \varphi(q_1)\times\\
&
\times\prodl_{j=1}^n\left[ p_A(t\lambda/n,q_{jN},q_{1+jN})\prodl_{k=1}^{N-1}p_A(s_j/(N-1),q_{k+(j-1)N},q_{k+1+(j-1)N}) \right]\times \\
&
\times\prodl_{j=1}^n \prodl_{k=1}^{N-1}dq_{k+(j-1)N}\eta^0_{t/n}(ds_j),
\end{align*}

\noindent (ii) Under  assumptions and notations of Proposition \ref{prop:BGS} and Theorem \ref{thm:ChFam_known bdd Levy} the  family $(\mathcal{F}_\mu(t))_{t\geq0}$ of Theorem \ref{thm:ChFam_known bdd Levy}, constructed with the help of the family $(F(t))_{t\geq0}$ given in \eqref{eq:F(t) BGS},  is Chernoff equivalent to the semigroup $(T^f_t)_{t\geq 0}$ and  has the following explicit view:
\begin{align*}
&\mathcal{F}_\mu(0):=\id\quad\text{ and   } \quad \mathcal{F}_\mu(t)\ffi(q):=\frac{\exp(-t(\sigma +\lambda C(q)))}{\sqrt{\det A(q)(2\pi t\lambda)^{d}}} \times\\
&
\times \intl_{\mathbb{R}^d} \exp\bigg(\frac{-A^{-1}(q)(q-q_{m(t)+1}+t\lambda B(q))\cdot(q-q_{m(t)+1}+t\lambda B(q))}{2t\lambda}\bigg)\Bigg(\ffi(q_{m(t)+1})+\\
&
+t\intl_{0+}^\infty \Bigg[ \,\intl_{\cR^{dm(t)}}\exp\bigg( -\frac{s}{m(t)}\suml_{k=1}^{m(t)}C(q_{k+1})  \bigg)        \prodl_{k=1}^{m(t)} \bigg(\det A(q_{k+1})(2\pi t\lambda)^{d} \bigg)^{-1/2}    \times\\
&
\times \exp\Bigg( -\suml_{k=1}^{m(t)} \frac{A^{-1}(q_{k+1})\Big(q_{k+1}-q_{k}+\frac{s}{m(t)}B(q_{k+1})\Big)\cdot\Big(q_{k+1}-q_{k}+\frac{s}{m(t)}B(q_{k+1})\Big)}{2s/m(t)}  \Bigg)\times\\
&
\times {\varphi}(q_1)dq_1\ldots dq_{m(t)} -\ffi(q_{m(t)+1})\Bigg]\mu(ds)\Bigg)dq_{m(t)+1}.
\end{align*}
\end{theorem}

\subsection{Approximation of subordinate   diffusions on a star graph }

Consider a star graph $\Gamma$ with vertex $v$ and $d\in \mathbb{N}$ external edges $l_1,\ldots, l_d$. Let  $\rho$ be the metric on $\Gamma$  induced by the isomorphism $l_k\cong[0,+\infty)$;  $\Gamma^o=\Gamma\setminus\{ v\}=\sqcup_{k=1}^d l^o_k$, $l^o_k\cong(0,+\infty)$. For each point $\xi\in\Gamma$ one has   $\xi\in l^o_k\Rightarrow \xi=(k,x)$, where  $x=\rho(\xi,v)>0$. For each function $\ffi:\Gamma\to\cR$ define  $\varphi_k(x):=\varphi(\xi)\big|_{\xi\in l^o_k}$ and   $\intl_{\Gamma}\varphi(\xi)d\xi:=\suml_{k=1}^d\intl_0^\infty\varphi_k(x)dx$.
Let $C_\infty(\Gamma)$ be the Banach space of continuous functions on $\Gamma$ vanishing at infinity equipped with the sup-norm $\|\cdot\|_\infty$. Let $C^2_\infty(\Gamma)=\{\varphi\in  C_\infty(\Gamma):\, \varphi\in C^2_\infty(\Gamma^o),\varphi''\,\text{extends to}\, \Gamma \,\text{as a function in}\, C_\infty(\Gamma)\}$. Let $\delta_v$ be the Dirac delta-measure concentrated at the vertex $v$.  Let $\rho_v(\xi,\zeta):=\rho(\xi,v)+\rho(v,\zeta)$ for  all $\xi,\zeta\in \Gamma$.
Let $1_{k}(\xi)=1$ if $\xi\in l^o_k$, $1_{k}(\xi)=0$ if $\xi\notin l^o_k$.
 Let $g(t,z)=(2\pi t)^{-1/2}\exp\big\{\frac{-z^2}{2t}\big\}$.
Define
$
 p(t,\xi,\zeta)=\suml_{k=1}^d1_k(\xi)1_k(\zeta)g(t,\rho(\xi,\zeta))$,
 $p_v(t,\xi,\zeta)=\suml_{k=1}^d1_k(\xi)1_k(\zeta)g(t,\rho_v(\xi,\zeta))$ and
 $p^D(t,\xi,\zeta)=p(t,\xi,\zeta)-p_v(t,\xi,\zeta)$.
 Let $a$, $c$, $b_k\in[0,1]$, $k=1,\ldots,d$, $a\neq 1$ and
$
a+c+\suml_{k=1}^db_k=1.
$
Consider an operator $L_0$ on  $C_\infty(\Gamma)$  with $\D(L_0)=\{ \varphi\in C^2_\infty(\Gamma):\, a\varphi(v)+\frac{c}{2}\varphi''(v)=\suml_{k=1}^d b_k\varphi'_k(v)  \}$ and $L_0\varphi=\frac12\varphi''$ for all $\varphi\in\D(L_0)$. Due to results of V. Kostrykin, Ju. Potthoff and R. Schrader \cite{MR2927703} the following statement is true.
\begin{proposition}\label{prop:Graph}
The operator  $(L_0,\D(L_0))$ is the generator of a strongly continuous semigroup  $(T^0_t)_{t\geq 0}$ on the space $C_\infty(\Gamma)$ and for each $\varphi\in C_\infty(\Gamma)$ one has $
T^0_t\varphi(\xi)=\intl_{\Gamma}\varphi(\zeta)P(t,\xi,d\zeta)
$, where the transition kernel $P(t,\xi,d\zeta)$  is given explicitly  by the following formulae:

\noindent for the case $ a+c\in(0,1)$ with $\, w_k=\frac{b_k}{1-a-c}\,$, $\,\beta=\frac{a}{1-a-c}\,$, $\,\gamma=\frac{c}{1-a-c} \,$ and
$$
g_{\beta,\gamma}(t,z)=\frac{1}{\gamma^2}(2\pi t)^{-1/2}\intl_0^t\frac{s+\gamma z}{(t-s)^{3/2}}    \exp\bigg\{-\frac{(s+\gamma z)^2}{2\gamma^2(t-s)}\bigg\}e^{-\beta s/\gamma}  ds,
$$
\begin{align}\label{general P}
&P(t,\xi,d\zeta)=\\
&
=p^D(t,\xi,\zeta)d\zeta+\suml_{k,j=1}^d 1_k(\xi)1_j(\zeta)2w_jg_{\beta,\gamma}(t, \rho_v(\xi,\zeta))d\zeta+\gamma g_{\beta,\gamma}(t,\rho(\xi,v))\delta_v(d\zeta);\nonumber
\end{align}
\noindent for the case  $a+c=0$ with $w_k=b_k$
\begin{align}
&P(t,\xi,\zeta)=p^D(t,\xi,\zeta)d\zeta+\suml_{k,j=1}^d 1_k(\xi)1_j(\zeta)2w_jg(t, \rho_v(\xi,\zeta))d\zeta;
\end{align}
\noindent for the case  $a+c=1$ with $a=\frac{\beta}{1+\beta},\quad c=\frac{1}{1+\beta}$
\begin{align}
&P(t,\xi,\zeta)=p^D(t,\xi,\zeta)d\zeta-\left(\intl_0^t e^{-\beta(t-s)}\frac{\rho(\xi,v)}{\sqrt{2\pi s^3}}\exp\left\{ -\frac{\rho(\xi,v)^2}{2s} \right\}ds,   \right)\delta_v(d\zeta).
\end{align}
\end{proposition}
The heat kernel  $P(t,\xi,d\zeta)$ in \eqref{general P} is the transition kernel of the process of Brownian motion on $\Gamma$ constructed   by killing (after an exponential holding  time  with the rate $\beta$ at the  vertex) the Walsh process (the analogue of the reflected Brownian motion)  with sticky vertex with stickness parameter $\gamma$ (see \cite{MR2927703} for the detailed exposition).

Let now $A(\cdot)\in C(\Gamma)$ and there exist   $ \underline{\alpha}, \overline{\alpha}\in(0,+\infty)$  such that $\underline{\alpha}\le A(\xi)\le\overline{\alpha}$ for all $\xi\in \Gamma$.  Let $B(\cdot)\in C_b(\Gamma)$ with $B(v)=0$. %%, $B(\xi)\ge0$ for all $\xi\in \Gamma$.
Let $C(\cdot)\in C_b(\Gamma)$ and $C(\xi)\geq0$ for all $\xi\in\Gamma$. As before let $a$, $c$, $b_k\in[0,1]$, $k=1,\ldots,d$ with $a\neq 1$ and
$
a+c+\suml_{k=1}^db_k=1.
$
As in \cite{Yana_QMATH12} consider an operator $L$ such that for all  $\varphi\in\D(L):=\big\{ \varphi\in C^2_\infty(\Gamma):\, a\varphi(v)+\frac{c}{2}\varphi''(v)=\sum_{k=1}^d b_k\varphi'_k(v) \big\}$  one has
$$
 L\varphi(\xi):=A(\xi)\varphi''(\xi)+B(\xi)\varphi'(\xi)-C(\xi)\varphi(\xi).
$$
Then the  operator  $(L,\D(L))$ is the generator of a  strongly continuous contraction semigroup $(T_t)_{t\geq 0}$ on the space $C_\infty(\Gamma)$. Let now the semigroup $(T^f_t)_{t\geq 0}$ be subordinate   to the semigroup $(T_t)_{t\geq 0}$ with respect to a given convolution semigroup $(\eta_t)_{t\geq 0}$ associated to a Bernstein function $f$ defined by a triplet $(\sigma,\lambda,\mu)$. The statement below follows immediately from Proposition \ref{prop:Graph} and Theorems  \ref{TH5.1}, \ref{thm:Technomag mult pert},  \ref{thm:ChFam_knownSub-r} and   \ref{thm:ChFam_known bdd Levy}.

\begin{theorem}
(i) Under notations of Proposition \ref{prop:Graph} consider a family $(\widetilde{F}(t))_{t\geq 0}$ on $C_\infty(\Gamma)$ defined by
$$
\widetilde{F}(t)\varphi(\xi):= \intl_{\Gamma}\varphi(\zeta)P(A(\xi)t,\xi,d\zeta).
$$
Consider  a family $(S_t)_{t\geq 0}$  on $C_\infty(\Gamma)$ defined by
$$
S_t\varphi(\xi):= \varphi(\xi+tB(\xi)):=\left\{
\begin{array}{lll}
\varphi_k(x+tB_k(x)),&\xi=(k,x),& x+tB_k(x)>0,\\
\varphi(v),&\xi=(k,x),& x+tB_k(x)\le0,\\
\varphi(v),&\xi=v.&
\end{array}
\right.
$$
  Consider  a family $(e^{-tC})_{t\geq 0}$ with $(e^{-tC}\varphi)(\xi):=e^{-tC(\xi)}\varphi(\xi)$.  By  Proposition~\ref{prop:Graph},  Theorem~\ref{TH5.1}  and Theorem~\ref{thm:Technomag mult pert}  the family $(F(t))_{t\geq 0}$ with  $F(t):=e^{-tC}\circ S_t\circ \widetilde{F}(t)$, i.e.
\begin{equation}\label{eq:F(t) Graph}
  { F(t)\varphi(\xi)=e^{-tC(\xi)}\intl_{\Gamma}\varphi(\zeta)P(A(\xi+tB(\xi))t,\xi+tB(\xi),d\zeta)},
\end{equation}
is Chernoff equivalent to the semigroup $(T_t)_{t\geq 0}$ on the space $C_\infty(\Gamma)$ generated by $(L,\D(L))$. 
  %% And hence the Feynman formula \eqref{FF} is valid locally uniformly with respect to   $t\ge0$ in $\mathcal{L}(C_\infty(\Gamma))$.
  
  \noindent (ii)  Under assumptions and notations of Theorem \ref{thm:ChFam_knownSub-r} the family $(\mathcal{F}(t))_{t\geq0}$, constructed with the help of $(F(t))_{t\geq 0}$ given by \eqref{eq:F(t) Graph},  is Chernoff equivalent to the semigroup $(T^f_t)_{t\geq 0}$ and has the following explicit view: $\mathcal{F}(0):=\id$ and %%with $\xi_{m(t)+2}:=\xi$
\begin{align*}
&\mathcal{F}(t)\ffi(\xi):=e^{-\sigma t}\intl_{0+}^\infty\intl_{\Gamma^{m(t)+1}}\exp\left(-\lambda t C(\xi)-\frac{s}{m(t)}\suml_{k=1}^{m(t)}C(\xi_{k+1})   \right)\ffi(\xi_1)\times\\
&
\times \prodl_{k=1}^{m(t)}P\big(A(\xi_{k+1}+(s/m(t))B(\xi_{k+1}))s/m(t),\xi_{k+1}+(s/m(t))B(\xi_{k+1}),d\xi_k\big)\times\\
&
\times P\big(A(\xi+t\lambda B(\xi))t\lambda,\xi+t\lambda B(\xi),d\xi_{m(t)+1}\big)\, d\eta^0_t(ds).
\end{align*}  
Therefore, 
for each $\ffi\in C_\infty(\Gamma)$ the following Feynman formula is true  with  $N:=1+m(t/n)$ and $\xi_{1+nN}:=\xi$  
 \begin{align*}
&T^f_t\ffi(\xi)=\\
&
=\liml_{n\to\infty}e^{-\sigma t}\,\intl_{(0,\infty)^n}\,\intl_{\Gamma^{nN}}\exp\bigg(-\suml_{j=1}^n\bigg[\frac{\lambda t}{n} C(\xi_{1+jN})-\frac{s_j}{N-1}\suml_{k=1}^{N-1}C(\xi_{k+1+(j-1)N})  \bigg] \bigg)\times\\
&
\times \ffi(\xi_1)\prodl_{j=1}^n\Bigg[P\big(A(\xi_{1+jN}+(t\lambda/n) B(\xi_{1+jN}))t\lambda/n,\xi_{1+jN}+(t\lambda/n)B(\xi_{1+jN}),d\xi_{jN}\big)\times\\
&
 \times\prodl_{k=1}^{N-1}P\bigg(A\big(\xi_{1+k+(j-1)N}+s_jB(\xi_{1+k+(j-1)N})/(N-1)\big)s_j/(N-1),\\
 &
 \quad\quad\quad\quad\quad\xi_{1+k+(j-1)N}+s_jB(\xi_{1+k+(j-1)N})/(N-1),\,d\xi_{k+(j-1)N}\bigg) \, d\eta^0_{t/n}(ds_j)\Bigg].
\end{align*}

  \noindent (iii)  Under assumptions and notations of Theorem \ref{thm:ChFam_known bdd Levy} the family  $(\mathcal{F}_\mu(t))_{t\geq0}$, constructed with the help of $(F(t))_{t\geq 0}$ given by \eqref{eq:F(t) Graph},  is Chernoff equivalent to the semigroup $(T^f_t)_{t\geq 0}$ and has the following explicit view: $\mathcal{F}_\mu(0):=\id$ and
  \begin{align*}
&\mathcal{F}_\mu(t)\ffi(\xi):=\\
&
=e^{-\sigma t-\lambda t C(\xi)}
\intl_\Gamma\Bigg[\ffi(\xi_{m(t)+1})+t\intl_{0+}^\infty
\Bigg(\,\,\intl_{\Gamma^{m(t)}}\exp\bigg(-\frac{s}{m(t)}\suml_{k=1}^{m(t)}C(\xi_{k+1})   \bigg)\ffi(\xi_1)\times\\
&
\times \prodl_{k=1}^{m(t)}P\big(A(\xi_{k+1}+(s/m(t))B(\xi_{k+1}))s/m(t),\xi_{k+1}+(s/m(t))B(\xi_{k+1}),d\xi_k\big)-\\
&
-\ffi(\xi_{m(t)+1})\Bigg)\mu(ds)\Bigg] P\big(A(\xi+t\lambda B(\xi))t\lambda,\xi+t\lambda B(\xi),d\xi_{m(t)+1}\big).
\end{align*} 
\end{theorem}

\subsection{Approximation of subordinate   diffusions in a Riemannian manifold}\label{subsec:Riemann}
%\cite{MR2276523}
 Let $\Gamma$ be a compact connected Riemannian manifold of class $C^\infty$  without boundary, $\dim \Gamma=d$.
 Let $\rho_\Gamma$  be the distance in $\Gamma$ generated by the Riemannian metric of $\Gamma$.  Let $\vol_\Gamma$  be the corresponding Riemannian volume measure on $\Gamma$. Assume also  that $\Gamma$ is isometrically embedded into a Riemannian manifold $G$ of dimension $\kappa$ and into  the Euclidean space $\cR^D$.  Let $\Phi$ be a  $C^\infty$-smooth isometric embedding of $\Gamma$ into $\cR^D$ and $\Phi_G$ be a  $C^\infty$-smooth isometric embedding of $\Gamma$ into $G$. Let $\rho_G$ be the distance in $G$ generated by the Riemannian metric of $G$. Consider the Laplace--Beltrami operator $\Delta_\Gamma:=-\tr\nabla_\Gamma^2$, where $\nabla_\Gamma$ is the Levi-Civita connection on $\Gamma$. The closure of $(\Delta_\Gamma, C^3(\Gamma))$ generates the \emph{heat semigroup}, i.e. the strongly continuous contraction  semigroup $(e^{\frac{t}{2}\Delta_\Gamma})_{t\geq0}$ on the space $C(\Gamma)$. Due to results of O.G. Smolyanov, H. v. Weizs\"{a}cker and O. Wittich (see \cite{MR2276523}, Sect. 5) the following is true.
 \begin{proposition}\label{prop:SWW}
 For all $t>0$, $x,y\in \Gamma$ consider pseudo-Gaussian kernels
 \begin{align*}
 & K_1(t,x,y):=(2\pi t)^{-d/2}\exp\left(- \frac{\rho_\Gamma(x,y)^2}{2t}   \right), \\
 &
 K_2(t,x,y):=(2\pi t)^{-d/2}\exp\left( -\frac{\rho_G(\Phi_G(x),\Phi_G(y))^2}{2t}   \right),\\
 &
 K_3(t,x,y):=(2\pi t)^{-d/2}\exp\left(- \frac{|\Phi(x)-\Phi(y)|^2}{2t}   \right).
 \end{align*}
 For each kernel $K_i$, $i=1,2,3$, define the family   $(F_i(t))_{t\geq0}$, $i=1,2,3$, of  contractions on $C(\Gamma)$ by
 \begin{align*}
& F_i(0):=\id \quad\text{ and for each }\quad \ffi\in C(\Gamma)\quad\text{ and each }\quad t>0\\
 &
 F_i(t)\ffi(x):=\frac{\int_\Gamma K_i(t,x,y)\ffi(y)\vol_\Gamma(dy)}{\int_\Gamma K_i(t,x,y)\vol_\Gamma(dy)}.
 \end{align*}
 Then each  family $(F_i(t))_{t\geq0}$, $i=1,2,3$, is Chernoff equivalent to the heat semigroup $(e^{\frac{t}{2}\Delta_\Gamma})_{t\geq0}$ on the space $C(\Gamma)$ with $\lim_{t\to0}\big\|
\frac{F_i(t)\ffi-\ffi}{t}-\frac12\Delta_\Gamma\ffi\big\|_\infty=0$ for all $\ffi\in C^3(\Gamma)$.
 \end{proposition}

Let now $A(\cdot)\in C_b(\Gamma)$ be a strictly positive  function, $B(\cdot):\Gamma\to T\Gamma$ be a bounded vector field of class $C^1(\Gamma)$ and $C(\cdot)\in C_b(\Gamma)$ be a nonnegative function.  Denote the inner product of vectors $u(x)$ and $v(x)$  in the tangent space $T_x\Gamma$ as $u(x)\cdot v(x)$.  As in \cite{MR2423533} define the family $(S_t)_{t\geq0}$ on $C(\Gamma)$ by
\begin{align*}
S_t\ffi(x):=\ffi(\gamma^x(t)),
\end{align*}
where $\gamma^x(\cdot)$ is a geodesic with starting point $x$  (i.e. $\gamma^x(0)=x$) and the direction vector $B(x)$ (i.e. $\dot{\gamma}^x(0)=B(x)$). Since the manifold is smooth and compact, and the vector field $B$ is smooth, the family $(S_t)_{t\geq0}$ is well defined as a family of strongly continuous contractions on $C(\Gamma)$ and $\lim_{t\to0}\left\|t^{-1}(S_t\ffi-\ffi)-B\cdot\nabla_\Gamma\ffi\right\|_\infty=0$ for all $\ffi\in C^3(\Gamma)$. Further, consider the operator $L$ defined on the set $C^3(\Gamma)$ by
\begin{align*}
L\ffi(x):=\frac12 A(x)\Delta_\Gamma \ffi(x)+B(x)\cdot \nabla_\Gamma\ffi(x)-C(x)\ffi(x).
\end{align*}
Then the closure of $(L, C^3(\Gamma))$ generates a strongly continuous contraction semigroup $(T_t)_{t\geq0}$ on $C(\Gamma)$. Let now the semigroup $(T^f_t)_{t\geq 0}$ be subordinate   to the semigroup $(T_t)_{t\geq 0}$ with respect to a given convolution semigroup $(\eta_t)_{t\geq 0}$ associated to a Bernstein function $f$ defined by a triplet $(\sigma,\lambda,\mu)$. The statement below follows immediately from Proposition \ref{prop:SWW}, Theorems \ref{TH5.1}, \ref{thm:Technomag mult pert}, \ref{thm:ChFam_knownSub-r}, \ref{thm:ChFam_known bdd Levy} and results of \cite{MR2423533}. 
\begin{theorem}
(i) For each of the the family   $(F_i(t))_{t\geq0}$, $i=1,2,3$, of Proposition \ref{prop:SWW} define  as in  Theorem \ref{thm:Technomag mult pert} the  families $(\widetilde{F}_i(t))_{t\geq0}$, $i=1,2,3$, of  contractions on $C(\Gamma)$ by 
$$
\widetilde{F}_i(t)\ffi(x):=(F_i(A(x)t)\ffi)(x).
$$
Further, define the families $(\widehat{F}_i(t))_{t\geq0}$, $i=1,2,3$, by
\begin{align*}
 \widehat{F}_i(t)\ffi(x):&=\left( e^{-tC}\circ S_t\circ \widetilde{F}_i(t)\right)\ffi(x) =\\
 &
 =\frac{\int_\Gamma e^{-tC(x)}K_i(A(\gamma^x(t))t,\gamma^x(t),y)\ffi(y)\vol_\Gamma(dy)}{\int_\Gamma K_i(A(\gamma^x(t))t,\gamma^x(t),y)\vol_\Gamma(dy)}.
 \end{align*}
 Then by Proposition \ref{prop:SWW},  Theorem \ref{TH5.1} and Theorem \ref{thm:Technomag mult pert}   the families $(\widehat{F}_i(t))_{t\geq0}$, $i=1,2,3$, are Chernoff equivalent to the semigroup $(T_t)_{t\geq0}$ on $C(\Gamma)$.
 
 \noindent (ii)   Under assumptions and notations of  Theorem  \ref{thm:ChFam_knownSub-r} the families  $(\mathcal{F}^i(t))_{t\geq0}$, $k=1,2,3$,  constructed as in Theorem \ref{thm:ChFam_knownSub-r} with the help of $(\widehat{F}_i(t))_{t\geq 0}$ given above,  are Chernoff equivalent to the semigroup $(T^f_t)_{t\geq 0}$ and have the following explicit view: $\mathcal{F}^i(0):=\id$ and 
 \begin{align*}
 &\mathcal{F}^i(t)\ffi(x)=e^{-\sigma t}\intl_{0+}^\infty\,\intl_{\Gamma^{m(t)+1}}\exp\bigg( -\lambda t C(x)-\frac{s}{m(t)}\suml_{k=1}^{m(t)} C(x_{k+1})   \bigg)\ffi(x_1)\times\\
   &
 \times \prodl_{k=1}^{m(t)} \frac{ K_i\big(A(\gamma^{x_{k+1}}(s/m(t)))s/m(t),\, \gamma^{x_{k+1}}(s/m(t)),\,x_{k}    \big)}{ \int_\Gamma  K_i\big(A(\gamma^{x_{k+1}}(s/m(t)))s/m(t),\, \gamma^{x_{k+1}}(s/m(t)),\,x_{k}    \big) \vol_\Gamma(dx_{k}) }\times\\
 &
 \times  \frac{ K_i\big(A(\gamma^{x}(\lambda t))\lambda t,\, \gamma^{x}(\lambda t),\,x_{m(t)+1}    \big)}{ \int_\Gamma  K_i\big(A(\gamma^{x}(\lambda t))\lambda t,\, \gamma^{x}(\lambda t),\,x_{m(t)+1}    \big) \vol_\Gamma(dx_{m(t)+1}) }\prodl_{k=1}^{m(t)+1}\vol_\Gamma(dx_k)\eta^0_t(ds).
 \end{align*}
 
 \noindent (iii) Under assumptions and notations of Theorem \ref{thm:ChFam_known bdd Levy} tha families $(\mathcal{F}^i_\mu(t))_{t\geq0}$, $k=1,2,3$,  constructed as in Theorem \ref{thm:ChFam_knownSub-r} with the help of $(\widehat{F}_i(t))_{t\geq 0}$ given above,  are Chernoff equivalent to the semigroup $(T^f_t)_{t\geq 0}$ and have the following explicit view: $\mathcal{F}^i_\mu(0):=\id$ and 
 \begin{align*}
& \mathcal{F}^i_\mu(t)\ffi(x)=\\
&
=e^{-\sigma t-\lambda t C(x)}\intl_\Gamma \Bigg(\ffi(x_{m(t)+1})+t\intl_{0+}^\infty \Bigg[\intl_{\Gamma^{m(t)}}\exp\bigg( -\frac{s}{m(t)}\suml_{k=1}^{m(t)}C(x_{k+1}) \bigg)\ffi(x_1)\times\\
 &
 \times \prodl_{k=1}^{m(t)} \frac{ K_i\big(A(\gamma^{x_{k+1}}(s/m(t)))s/m(t),\, \gamma^{x_{k+1}}(s/m(t)),\,x_{k}    \big)}{ \int_\Gamma  K_i\big(A(\gamma^{x_{k+1}}(s/m(t)))s/m(t),\, \gamma^{x_{k+1}}(s/m(t)),\,x_{k}    \big) \vol_\Gamma(dx_{k}) }\prodl_{k=1}^{m(t)}\vol_\Gamma(dx_k)-\\
 &
  - \ffi(x_{m(t)+1})\Bigg]\mu(ds)\Bigg)\frac{ K_i\big(A(\gamma^{x}(\lambda t))\lambda t,\, \gamma^{x}(\lambda t),\,x_{m(t)+1}    \big)\vol_\Gamma(dx_{m(t)+1})}{ \int_\Gamma  K_i\big(A(\gamma^{x}(\lambda t))\lambda t,\, \gamma^{x}(\lambda t),\,x_{m(t)+1}    \big) \vol_\Gamma(dx_{m(t)+1}) }.
 \end{align*}
\end{theorem}

%%{\bf Acknowledgements.}   The author thanks ........................... 

 \bibliographystyle{abbrv}
 \bibliography{yana15-bib}

\end{document}